\documentclass[11pt,letterpaper]{amsart}
\usepackage{amsthm}
\usepackage{amsmath}
\usepackage{amssymb}
\usepackage{comment}
\usepackage{tikz}
\usepackage{enumitem}
\usepackage{amsfonts}
\usepackage{leftindex}
\usepackage{mathtools}
\usepackage{hyperref}
\usepackage{cleveref}
\usepackage{todonotes}
\usepackage[utf8]{inputenc}
\usepackage[left=1.25in, right=1.25in, bottom=1.25in, top=1.25in]{geometry}

\DeclareMathOperator{\Supp}{Supp}

\DeclareMathOperator{\BP}{\mathsf{BP}}

\DeclareMathOperator{\SL}{SL}

\DeclareMathOperator{\bp}{\mathsf{bp}}
\DeclareMathOperator{\tbp}{\widetilde{\mathsf{bp}}}

\DeclareMathOperator{\cl}{cl}

\newcommand\precdot{\mathrel{\ooalign{$\prec$\cr
  \hidewidth\raise0.001ex\hbox{$\cdot\mkern0.6mu$}\cr}}}

\newtheorem{theorem}{Theorem}[section]
\newtheorem{def-prop}[theorem]{Definition-Proposition}
\newtheorem{prop}[theorem]{Proposition}
\newtheorem{conj}[theorem]{Conjecture}
\newtheorem{lemma}[theorem]{Lemma}

\newtheorem{cor}[theorem]{Corollary}

\theoremstyle{definition}
\newtheorem{ex}[theorem]{Example}

\newtheorem{defin}[theorem]{Definition}

\theoremstyle{remark}
\newtheorem{remark}{Remark}

\Crefname{remark}{Remark}{Remarks}
\Crefname{prop}{Proposition}{Propositions}

\begin{document}
\title[Billey--Postnikov posets]{Billey--Postnikov posets, rationally smooth Schubert varieties, and Poincar\'e duality}

\author{Christian Gaetz}
\address{Department of Mathematics, University of California, Berkeley, CA, USA.}
\email{\href{mailto:gaetz@berkeley.edu}{{\tt gaetz@berkeley.edu}}}

\author{Yibo Gao}
\address{Beijing International Center for Mathematical Research, Peking University, Beijing, China.}
\email{\href{mailto:gaoyibo@bicmr.pku.edu.cn}{{\tt gaoyibo@bicmr.pku.edu.cn}}}

\date{\today}

\begin{abstract}
\emph{Billey--Postnikov (BP) decompositions} govern when Schubert varieties $X(w)$ decompose as bundles of smaller Schubert varieties. We further develop the theory of BP decompositions and show that, in finite type, they can be recognized by pattern conditions and are indexed by the order ideals of a poset $\bp(w)$ that we introduce; we conjecture that this holds in any Coxeter group. We then apply BP decompositions to show that, when $X(w)$ is rationally smooth and $W$ simply laced, the Schubert structure constants $c_{u,v}^w$ satisfy a triangularity property, yielding a canonical involution on the Schubert cells of $X(w)$ respecting Poincar\'{e} duality. We also classify the rationally smooth Bruhat intervals in finite type (other than $E$) which admit generalized \emph{Lehmer codes}, answering questions and conjectures of Billey--Fan--Losonczy, Bolognini--Sentinelli, and Bishop--Mili\'{c}evi\'{c}--Thomas. Finally, we show that rationally smooth Schubert varieties in infinite type need not have Grassmannian BP decompositions, disproving conjectures of Richmond--Slofstra and Oh--Richmond.
\end{abstract}
\keywords{}
\maketitle

\section{Introduction}

Let $G$ be a complex reductive group (or, more generally, a Kac--Moody group) with Borel subgroup $B$. The $B$-orbit closures in the generalized flag variety $G/B$ are called \emph{Schubert varieties} $X(w)$ and are indexed by elements $w$ of the Weyl group $W$. These varieties are important for many reasons. For one, the classes $\{[X(w)] \mid w \in W\}$ give a basis for $H^*(G/B)$, the study of whose structure constants $c_{u,v}^w$ has long defined the field of \emph{Schubert calculus}. On the other hand, in geometric representation theory, \emph{Kazhdan--Lusztig theory} \cite{kazhdan-lusztig} relates singularities of the $X(w)$, characters of Verma modules of $\mathfrak{g}$, and distinguished bases of the Hecke algebra $\mathcal{H}(W;q)$.

The parabolic subgroups $P_J \supset B$ are indexed by subsets $J$ of the simple reflections $S$ generating $W$. Under the natural projection $\pi^J: G/B \twoheadrightarrow G/P_J$, it is known that $X(w)$ maps onto a $B$-orbit closure $X^J(w^J) \coloneqq \overline{Bw^JP_J/P_J} \subset G/P_J$; the $X^J(w^J)$ are also known as (parabolic) Schubert varieties. The fibers of $\pi^J |_{X(w)}$ in general vary in dimension. However, if they are of constant dimension, then all fibers are isomorphic to $X(w_J)$, so that $X(w)$ is the total space of a fiber bundle whose base and fiber are both smaller Schubert varieties; in this case the decomposition $w=w^Jw_J$ is called a \emph{Billey--Postnikov (BP) decomposition}. Aside from the geometric applications discussed below, BP decompositions have also found applications to enumeration, Bruhat combinatorics, and Kazhdan--Lusztig polynomials \cite{vertex-transitive, min-heights, richmond-slofstra-staircase}; see Oh--Richmond for a survey \cite{oh-richmond-chapter}.

For $G=\SL_n$, the varieties $G/P_J$ and $G/B$ are naturally identified with the classical (partial) flag varieties. In this setting, Lakshmibai--Sandhya \cite{lakshmibai-sandhya} famously characterized the smooth $X(w)$ as those for which $w$ avoids the \emph{permutation patterns} $3412$ and $4231$; the study of permutation patterns has since developed into an active subfield of combinatorics in its own right. This result was generalized to all finite types by Billey and Postnikov \cite{billey-postnikov, billey-A-B}, who (implicitly) introduced BP decompositions. Richmond and Slofstra \cite{richmond-slofstra-fiber-bundle} developed the theory of BP decompositions and used them to show that smooth Schubert varieties in finite type are iterated bundles of smooth Schubert varieties in $G/P$ for $P$ a \emph{maximal} parabolic subgroup, generalizing a result of Ryan and Wolper \cite{ryan, wolper}. BP decompositions for $P$ maximal are called \emph{Grassmannian} BP decompositions, since in this case $G/P$ is a Grassmannian when $G=\SL_n$. 

The correct level of generality for these results is when $X(w)$ is \emph{rationally smooth}, which can be defined in terms of the vanishing of certain local intersection cohomology groups. For our purposes, we take the equivalent condition, that the Kazhdan--Lusztig polynomial $P_{e,w}(q)$ is equal to $1$, as the definition. This has the added benefit of allowing us to define rational smoothness for an element $w$ in an arbitrary (not-necessarily-crystallographic) Coxeter group. In finite simply-laced type, it is a result of Carrell--Peterson and Carrell--Kuttler \cite{carrell-peterson, carrell-kuttler} that smoothness and rational smoothness are equivalent. For any $G$, it is also known \cite{bjorner-ekedahl} that rational smoothness is equivalent to the palindromicity of the Poincar\'{e} polynomial of $X(w)$. 

We now describe our main results. The first of these identify surprising new structure in the BP decompositions of any (not-necessarily rationally smooth) $w$. We show, in finite type, how to detect that $J$ induces a BP decomposition of $X(w)$ using \emph{$J$-star patterns}, which we introduce. And we show that the set of $J$ inducing a BP decomposition is closed under union and intersection, so that these decompositions are indexed by the order ideals of the \emph{Billey--Postnikov poset} $\bp(w)$ that we construct. We then return to the setting of rationally smooth $w$ and prove several kinds of results about Schubert structure constants $c_{u,v}^w$, about \emph{generalized Lehmer codes}, and about the failure of the Richmond--Slofstra result in infinite type.

\subsection{Patterns for BP decompositions and BP posets}

We first characterize BP decompositions by the avoidance of what we call \emph{$J$-star patterns}; see \Cref{sub:characterization} for precise definitions. This generalizes a result of Alland--Richmond \cite{alland-richmond-BP-pattern} which applies for $W=S_n$ when $J$ is maximal. We write $\BP(w)$ for the set of $J \subset S$ inducing a BP decomposition; see \Cref{def:BP} for a Coxeter-theoretic criterion.

\begin{theorem}
\label{thm:intro-bp-pattern-characterization}
Let $W$ be a finite Weyl group and $w\in W$. Then $J \in \BP(w)$ if and only if $w$ avoids all $J$-star patterns.
\end{theorem}

We apply \Cref{thm:intro-bp-pattern-characterization} to prove \Cref{thm:intro-union-intersection} below. It is not clear how one could prove \Cref{thm:intro-union-intersection} directly from the geometric or Coxeter-theoretic characterizations of BP decompositions.

\begin{theorem}
\label{thm:intro-union-intersection}
Let $W$ be a finite Weyl group and $w\in W$. If $J, K \in \BP(w)$, then we have $J \cap K, J \cup K \in \BP(w)$. In particular, $\BP(w)$ is a distributive lattice. 
\end{theorem}

\Cref{thm:intro-union-intersection} implies that there is a set partition $\mathcal{S}_w$ of $S$ and a partial order on its blocks, which we call the \emph{reduced BP poset} $\bp(w)$, such that the elements of $\BP(w)$ are precisely the unions of the blocks lying in the order ideals of $\bp(w)$.

We can use \Cref{thm:intro-bp-pattern-characterization} to characterize when $w$ admits a chain of \emph{maximal} BP decompositions and therefore when $X(w)$ is an iterated bundle of Grassmannian Schubert varieties. This recovers a classification due to Alland and Richmond for $W=S_n$ \cite{alland-richmond-BP-pattern} and includes the well-studied case of smooth Schubert varieties. When $w$ admits such a chain, all blocks of $\mathcal{S}_w$ are singletons, so $\bp(w)$ is a partial order on $S$. The poset $\bp(w)$ also generalizes aspects of the \emph{staircase diagram} construction \cite{richmond-slofstra-staircase} which applies in the smooth case.

We conjecture that \Cref{thm:intro-union-intersection} in fact holds in any Coxeter group.

\begin{conj}
\label{conj:union-intersection}
Let $W$ be any Coxeter group and let $w \in W$. If $J, K \in \BP(w)$, then $J \cap K, J \cup K \in \BP(w)$. 
\end{conj}

In \Cref{sec:rank-three}, we prove \Cref{conj:union-intersection} for all $W$ of rank at most three.

\subsection{Poincar\'{e} duality in smooth Schubert varieties}

The Schubert variety $X(w)$ is paved by open affine Schubert cells $\Omega_v \coloneqq BvB/B$ for $v \leq w$ in \emph{Bruhat order} on $W$. The geometrical context of the BP poset can be seen as follows:
\begin{prop}
\label{prop:intro-linear-extension}
Let $W$ be a finite Weyl group and let $w \in W$ be a rationally smooth element. Any linear extension of the BP poset $\bp(w)$ determines a decomposition of $X(w)$ as an iterated bundle of rationally smooth Grassmannian Schubert varieties. 
\end{prop}

A handy application of the technology of the BP poset is that we can determine many Schubert structure constants under a rationally smooth element. We write $c_{u,v}^w$ for the positive integer structure constants for multiplication of the Schubert basis of $H^*(G/B)$; that is, $[X(w_0u)][X(w_0v)]=\sum_w c_{u,v}^w [X(w_0w)]$. We write $[e,w]_k$ for the elements of length $k$ in the Bruhat interval $[e,w]$, which index the Schubert basis of $H^{2k}(G/B)$.

\begin{theorem}\label{thm:intro-structure-constant-matrix}
Let $W$ be a finite Weyl group of simply-laced type, and let $w\in W$ be a (rationally) smooth element. For any $0\leq k\leq \ell(w)$, we can linearly order $[e,w]_k$ and $[e,w]_{\ell(w)-k}$ so that the matrix $(c_{u,v}^w)$, whose rows are indexed by $u\in[e,w]_k$ and columns are indexed by $v\in[e,w]_{\ell(w)-k}$, is upper unitriangular.
\end{theorem}
\noindent \Cref{thm:intro-structure-constant-matrix} is proven in \Cref{sec:poincare-dual}.

A well-known characterization of rational smoothness for $w\in W$ is that the Bruhat interval $[e,w]$ is rank-symmetric. It is therefore natural to seek, for each rationally smooth $w\in W$ and $0\leq k\leq \ell(w)$, a bijection between $[e,w]_k$ and $[e,w]_{\ell(w)-k}$. In type $A$, for example, many such bijections can be constructed in an \emph{ad hoc} manner via the factorization of the rank generating function of $[e,w]$ (see, e.g. \cite{gasharov}). An important and surprising corollary to \Cref{thm:intro-structure-constant-matrix} is that in finite simply-laced type, there is a \emph{canonical} such bijection:
\begin{cor}
\label{cor:intro-bijection}
Let $W$ be a finite Weyl group of simply-laced type, and let $w\in W$ be a (rationally) smooth element. For all $0 \leq k \leq \ell(w)$, there is a \emph{unique} bijection $\phi:[e,w]_k \to [e,w]_{\ell(w)-k}$ such that $c_{u, \phi(u)}^w \neq 0$ for all $u \in [e,w]_k$.
\end{cor}
We note that \Cref{thm:intro-structure-constant-matrix} and \Cref{cor:intro-bijection} are not true in multiply-laced type (see \Cref{remark:F4-bijection-not-canonical}).

\subsection{Generalized Lehmer codes}

As a result of the affine paving by Schubert cells, we can express the Poincar\'{e} polynomial $\mathcal{P}(w) \coloneqq \sum_{i=0}^{\ell(w)} \dim(H^{2i}(X(w)))q^i$ of $X(w)$ as $\mathcal{P}(w)=\sum_{v \leq w} q^{\ell(v)}$; we can take this as the definition of $\mathcal{P}(w)$ for $W$ non-crystallographic. It is a classical fact that for $W$ a finite Coxeter group with rank $r$ and longest element $w_0$ we have 
\begin{equation}
\label{eq:degrees}
\mathcal{P}(w_0) = \prod_{i=1}^{r} [d_i]_q,
\end{equation}
where the $d_1,\ldots,d_{r}$ are integer invariants called the \emph{degrees} of $W$ and $[d]_q$ denotes the $q$-integer $1+q+\cdots+q^{d-1}$.

For $W=S_n$, the formula (\ref{eq:degrees}) can be realized explicitly via the \emph{Lehmer code}, which can be interpreted as giving an order preserving bijection from a product $C_{d_1} \times \cdots \times C_{d_r}$ of chains of cardinalities $d_1,\ldots,d_r$ to Bruhat order on $W$. Gasharov showed \cite{gasharov} that, far less obviously, for $W=S_n$, the Bruhat interval $[e,w]$ below any rationally smooth element $w$ also admits an order preserving bijection from a product of chains (which we will henceforward just call a Lehmer code for $[e,w]$). Billey--Fan--Losonczy conjectured that Gasharov's result could be extended to all finite Weyl groups.

\begin{conj}[Billey--Fan--Losonczy \cite{billey-fan-losonczy}]
\label{conj:billey-fan-losonczy}
Let $W$ be a finite Weyl group and let $w \in W$ be rationally smooth, then $[e,w]$ admits a Lehmer code.
\end{conj}

\Cref{conj:billey-fan-losonczy} was resolved in the affirmative by Billey in types $A$ and $B$ \cite{billey-A-B}. In \cite{cubulation} Bishop--Mili\'{c}evi\'{c}--Thomas report on some computations showing several cases, namely $F_4, H_4,$ and $E_6$, in which $[e,w_0]$ does \emph{not} have a Lehmer code (there under the name \emph{cubulation}); the authors also ask to what extent these observations can be extended to rationally smooth intervals. Recent results of Bolognini--Sentinelli and Sentinelli--Zatti \cite{bolognini-sentinelli, sentinelli-zatti} construct a Lehmer code for $[e,w_0]$ in $D_n$ and show one does not exist for $F_4$; those authors likewise wonder about the general case.

In \Cref{thm:intro-lehmer-code} we resolve \Cref{conj:billey-fan-losonczy} in all types except $E$, and indeed also for the finite non-crystallographic groups. There are two classes of irreducible $W$: we resolve the conjecture in the affirmative for the infinite families as well as for $H_3$; and for the exceptional groups of types $H_4$ and $F_4,$ we show that Lehmer codes exist for rationally smooth $w$ \emph{except} $w_0$.

\begin{theorem}
\label{thm:intro-lehmer-code}
Let $W$ be an irreducible finite Coxeter group not of type $E$.
\begin{itemize}
    \item[(1)] If $W$ is of type $A_n, B_n,D_n,I_2(m),$ or $H_3$ then every rationally smooth element admits a Lehmer code;
    \item[(2)] If $W$ is of type $H_4$ or $F_4$, then a rationally smooth element $w$ admits a Lehmer code if and only if $w \neq w_0$.
\end{itemize}
\end{theorem}

Our proof of \Cref{thm:intro-lehmer-code} in \Cref{sec:lehmer} relies on a study of BP decompositions and the Lehmer codes that they may induce. A few computations (which might be tractable with the appropriate implementation) would suffice to extend our methods to also cover type $E$; see \Cref{rmk:type-E}. In \Cref{thm:An-Ak-lehmer} we also extend Billey's result by showing that rationally smooth intervals in certain parabolic quotients of type $A$ admit Lehmer codes.

\subsection{Grassmannian BP decompositions in infinite type}
It is natural to hope that the result of Richmond--Slofstra expressing rationally smooth Schubert varieties as iterated bundles of Grassmannian Schubert varieties could be extended to infinite type. Whether this is possible was posed as a question by Oh and Richmond \cite[Q.~8.1]{oh-richmond-chapter} and an affirmative answer was conjectured by Richmond and Slofstra \cite[Conj.~6.5]{richmond-slofstra-fiber-bundle}.

\begin{conj}[Richmond--Slofstra \cite{richmond-slofstra-fiber-bundle}; Oh--Richmond \cite{oh-richmond-chapter}]
\label{conj:infinite-type-grassmannian-bp}
Let $W$ be any Coxeter group and suppose that $w \in W$ is rationally smooth. Then $w$ has a Grassmannian BP decomposition.
\end{conj}

\Cref{conj:infinite-type-grassmannian-bp} is known to hold when $W$ is finite \cite{richmond-slofstra-fiber-bundle, ryan, wolper}, of affine type $\widetilde{A}_n$ \cite{billey-crites}, or in a certain large family of groups of indefinite type \cite{richmond-slofstra-triangle}. In \Cref{thm:counterexample} we give the first known counterexample. 

\begin{theorem}
\label{thm:intro-counterexample}
\Cref{conj:infinite-type-grassmannian-bp} fails for the affine Weyl group $W$ of type $\widetilde{C}_2$.
\end{theorem}

We prove \Cref{thm:intro-counterexample} in \Cref{sec:infinite-type-counterexample}.

\section{Preliminaries}\label{sec:preliminaries}
\subsection{Coxeter groups and BP decompositions}
Let $W$ be a Coxeter group with simple reflection set $S$. We refer the reader to \cite{bjorner-brenti} for background and notation on Coxeter groups. In particular, we write $\ell$ for the length function, $\leq$ for the Bruhat order on $W$, and $\leq_L$ for left weak order. For $J \subset S$, we write $W_J$ for the parabolic subgroup generated by $J$, and $W^J$ for the set of minimal length representatives of the cosets $W/W_J$. The subgroup $W_J$ is a Coxeter group in its own right, with simple reflection set $J$. Each $w \in W$ has a unique parabolic decomposition $w=w^Jw_J$ with $w^J \in W^J$ and $w_J \in W_J$, and these elements satisfy $\ell(w)=\ell(w^J)+\ell(w_J)$. Parabolic decompositions for \emph{right} cosets are defined analogously, and written $w=\leftindex_J{w} \leftindex^J{w}$. The \emph{left descent set} of $w \in W$ is $D_L(w) \coloneqq \{ s \in S \mid \ell(sw)<\ell(w)\}$ and the \emph{support} $\Supp(w)$ of $w$ is $\{s \in S \mid s \leq w\}$; the right descent set $D_R$ is defined analogously. 

\begin{defin}[See \cite{billey-postnikov, richmond-slofstra-fiber-bundle}]
\label{def:BP}
A parabolic decomposition $w=w^Jw_J$ is called a \emph{Billey--Postnikov} decomposition, or \emph{BP} decomposition, if 
\[
\Supp(w^J)\cap J\subset D_L(w_J). 
\]
In this case, we also say that $w$ is \emph{BP at} $J$.
\end{defin}

Previous investigations of BP decompositions have often been satisfied to prove the \emph{existence} of a suitable BP decomposition for \emph{certain} (e.g. \emph{rationally smooth}) elements $w$. A perspective of this paper is that it is valuable to consider the set of \emph{all} BP decompositions of an arbitrary element $w \in W$. To this end, for any $w \in W$, we define
 \[\BP(w):=\{J\subset S\:|\: w=w^Jw_J \text{ is a Billey--Postnikov decomposition}\}.\]

We write $[v,w]$ for a closed Bruhat interval. If $v,w \in W^J$, we write $[v,w]^J$ for $[v,w] \cap W^J$. For $w \in W^J$ we write $\mathcal{P}^J(w)$ for the length generating function of the Bruhat interval below $w$:
\[
\mathcal{P}^J(w) \coloneqq \sum_{v \in [e,w]^J} q^{\ell(v)},
\]
and we write $\mathcal{P}(w)$ for $\mathcal{P}^{\emptyset}(w)$.

\begin{prop}[See \cite{richmond-slofstra-fiber-bundle}]
\label{prop:bp-iff-rgf}
Let $W$ be any Coxeter group, let $w \in W$, and let $J \subset S$. Then $J \in \BP(w)$ if and only if $\mathcal{P}(w)=\mathcal{P}^J(w^J)\mathcal{P}(w_J)$.
\end{prop}

\subsection{BP decompositions at disconnected subsets}
We say that $J\subset S$ is \emph{connected} if the induced subgraph of the Coxeter diagram on vertex set $J$ is connected. Two subsets $J,K\subset S$ are \emph{totally disconnected} if $J\cap K=\emptyset$ and every $s\in J$ commutes with every $s'\in K$. The following lemma is elementary.

\begin{lemma}\label{lem:totally-disconnected-parabolic}
Let $J$ and $K$ be totally disconnected. Then $w_{J\cup K}=w_Jw_K$.
\end{lemma}

\Cref{lem:totally-disconnected-BP} below will often allow us to reduce to the consideration of connected $J \subset S$.

\begin{lemma}\label{lem:totally-disconnected-BP}
Let $J$ and $K$ be totally disconnected. Then $J\cup K \in \BP(w)$ if and only if $J \in \BP(w)$ and $K \in \BP(w)$.
\end{lemma}
\begin{proof}
By \Cref{lem:totally-disconnected-parabolic}, $w_{J\cup K}=w_Jw_K$ and furthermore it is clear that $D_L(w_{J\cup K})=D_L(w_J)\sqcup D_L(w_K)$. At the same time, $w^{J\cup K}w_K=w^J$ is a length-additive expression, so $\Supp(w^{J\cup K})\cap J=\Supp(w^J)\cap J$. Similarly, $\Supp(w^{J\cup K})\cap K=\Supp(w^K)\cap K$. Thus, \[\Supp(w^{J\cup K})\cap(J\cup K)=\big(\Supp(w^J)\cap J\big)\sqcup\big(\Supp(w^K)\cap K\big)\]
and $\Supp(w^{J\cup K})\cap(J\cup K)\subset D_L(w_{J\cup K})=D_L(w_J)\cup D_L(w_K)$ if and only if both $\Supp(w^J)\cap J\subset D_L(w_J)$ and $\Supp(w^K)\cap K\subset D_L(w_K)$. 
\end{proof}

\subsection{Schubert varieties and rational smoothness}
For $v,w \in W$, and $J \subset S$ we write $P^J_{v,w}(q)$ for the corresponding \emph{parabolic Kazhdan--Lusztig polynomial} \cite{deodhar-parabolic, kazhdan-lusztig}, a polynomial in $\mathbb{N}[q]$ \cite{elias-williamson}.

\begin{theorem}[See \cite{carrell-peterson,deodhar-poincare-duality,kazhdan-lusztig}]
\label{thm:rs-equiv-defs}
Let $W$ be any Coxeter group, $J \subset S$, and $w \in W^J$. Then the following are equivalent:
\begin{itemize}
    \item[(RS1)] $P^J_{e,w}(q)=1$ and
    \item[(RS2)] the polynomial $\mathcal{P}^J(w)$ is palindromic: $q^{\ell(w)}\mathcal{P}^J(w)(q^{-1})=\mathcal{P}^J(w)(q)$.
\end{itemize}
\end{theorem}
\noindent We say an element $w$ satisfying these equivalent conditions is \emph{$J$-rationally smooth} (or just \emph{rationally smooth} if $J=\emptyset$).

The reason for this terminology is the following. Suppose $W$ is crystallographic (but not necessarily finite) and let $G$ be the corresponding Kac--Moody group. Let $T \subset B$ be a maximal torus inside a Borel subgroup of $G$. For $w \in W^J$, the \emph{Schubert variety} is $X^J(w) \coloneqq \overline{BwP_J}/P_J$, a finite-dimensional subvariety of the generalized flag variety $G/B$ (which may itself be an infinite-dimensional ind-variety). In this case, there is a third equivalent condition for rational smoothness:
\begin{itemize}
    \item[(RS3)] the local intersection cohomology $IH^i(X^J(w))_v$ at the $T$-fixed point $vP_J$ vanishes for all $W^J\ni v \leq w$ and $i>0$. 
\end{itemize}
Finally, if $W$ is of finite simply-laced type, then by a result of Carrell--Peterson \cite{carrell-peterson} we also have:
\begin{itemize}
    \item[(RS4)] $X^J(w)$ is a smooth variety.
\end{itemize}

\subsection{Root system conventions}
For much of the paper we work in the case $W$ is a finite Weyl group. Let $\Phi$ be a crystallographic root system of rank $r$, with a choice of positive roots $\Phi^+$ and the corresponding simple roots $\Delta=\{\alpha_1,\ldots,\alpha_r\}$. The \emph{root poset} on $\Phi^+$ is defined so that $\beta\leq\gamma$ if $\gamma-\beta$ can be written as a nonnegative linear (necessarily integral) combination of $\Delta$. For $\beta\geq\alpha\in\Delta$, we say that $\beta$ is \emph{supported on} $\alpha$. Let $W=W(\Phi)$ be the Weyl group of $\Phi$. For $w\in W$, its \emph{(right) inversion set} is $I(w)=\{\beta\in\Phi^+\:|\: w\beta\in\Phi\setminus \Phi^+\}$. The inversion set $I(w)$ is well known to be a \emph{biclosed set}:
\begin{itemize}
\item if $\beta\in I(w)$, $\gamma\in I(w)$, and $\beta+\gamma\in\Phi^+$, then $\beta+\gamma\in I(w)$; and
\item if $\beta \in \Phi^+ \setminus I(w)$, $\gamma \in \Phi^+ \setminus I(w)$, and $\beta+\gamma\in\Phi^+$, then $\beta+\gamma\notin I(w)$.
\end{itemize}

For $J\subset S$ we write $\Delta_J \subset \Delta$ for the corresponding simple roots and let $\Phi_J^+\subset\Phi^+$ be the set of positive roots that can be written as a linear combination of simple roots in $\Delta_J$. For any $w\in W$, we have $I(w^J)\cap \Phi_J^+=\emptyset$. 

We adopt the following conventions for root systems in classical types:
\begin{itemize}
\item Type $A_{n-1}$: $\Phi=\{e_i-e_j\:|\:1\leq i\neq j\leq n\}$, $\Phi^+=\{e_i-e_j\:|\:1\leq i<j\leq n\}$, $\Delta=\{e_i-e_{i+1}\:|\:1\leq i\leq n-1\}$.
\item Type $B_n$: $\Phi=\{\pm e_i\pm e_j,\ \pm e_i\:|\:1\leq i\neq j\leq n\}$, $\Phi^+=\{e_i\pm e_j,\ e_i\:|\:1\leq i<j\leq n\}$, $\Delta=\{e_i-e_{i+1}\:|\:1\leq i\leq n-1\}\cup\{e_n\}$.
\item Type $C_n$: $\Phi=\{\pm e_i\pm e_j,\ \pm 2e_i\:|\:1\leq i\neq j\leq n\}$, $\Phi^+=\{e_i\pm e_j,\ 2e_i\:|\:1\leq i<j\leq n\}$, $\Delta=\{e_i-e_{i+1}\:|\:1\leq i\leq n-1\}\cup\{2e_n\}$.
\item Type $D_n$: $\Phi=\{\pm e_i\pm e_j\:|\:1\leq i\neq j\leq n\}$, $\Phi^+=\{e_i\pm e_j\:|\:1\leq i<j\leq n\}$, $\Delta=\{e_i-e_{i+1}\:|\:1\leq i\leq n-1\}\cup\{e_{n-1}+e_n\}$,
\end{itemize}
where $e_i$ denotes the $i$-th standard basis vector in $\mathbb{R}^n$.

\section{The Billey--Postnikov poset}
In this section, we describe explicit ways to characterize Billey--Postnikov decompositions.  In \Cref{sub:characterization,sub:distributive,sub:bp-properties} we work in a finite Weyl group.

\subsection{Characterizations of Billey--Postnikov decompositions}\label{sub:characterization}
We now introduce $J$-star patterns, one of our main technical tools.

\begin{defin}\label{def:Jstar}
We say that $(\beta,c_1\gamma_1,\ldots,c_k\gamma_k)$ forms a \emph{$J$-star} if $\beta\in\Phi_J^+$, $\gamma_1,\ldots,\gamma_k\in\Phi^+\setminus\Phi_J^+$, and $c_1,\ldots,c_k$ are positive integers such that $\beta+\sum_{i\in I}c_i\gamma_i\in\Phi^+$ for all $I\subset[k]$. We say that $w\in W$ \emph{contains} the $J$-star $(\beta,c_1\gamma_1,\ldots,c_k\gamma_k)$ if $\beta\notin I(w)$ and $\beta+\sum_{i=1}^k c_i\gamma_i\in I(w)$.
\end{defin}

\begin{remark}
In finite crystallographic root systems the multiset $\{c_1,\ldots,c_k\}$ can only be $\{1\}$, $\{2\}$, $\{3\}$, $\{1,1\}$, $\{1,2\}$ or $\{1,1,1\}$.  
\end{remark}

We establish some technical lemmas before proving \Cref{thm:intro-bp-pattern-characterization}.

\begin{lemma}\label{lem:simple-Jstar}
Fix $\Delta_J\subset\Delta$. Let $\alpha\in \Delta_J$ be a simple root and let $\tau\in\Phi^+\setminus\Phi_J^+$ be supported on $\alpha$. Then either
\begin{enumerate}
\item there exists $\alpha_j\in \Delta_J$ such that $\tau-\alpha_j\in\Phi^+$ is also supported on $\alpha$; or
\item there exists a $J$-star $(\alpha,c_1\gamma_1,\ldots,c_k\gamma_k)$ such that $\alpha+\sum_{i=1}^k c_i\gamma_i=\tau$.
\end{enumerate}
\end{lemma}
\begin{proof}
Write $\tau=\sum d_i\alpha_i$ in the basis $\Delta$. Let us first consider the situation where $d_i\in\{0,1\}$ for all $\alpha_i\in\Delta$. This covers many cases, and, in particular, all of type $A$.

Note that $\beta=\sum_{i\in I}\alpha_i$ is a positive root if and only if $I$ is a connected subgraph of the Dynkin diagram of $\Delta$ (see, e.g. \cite[Ch.VI \S 1]{bourbaki}). Now suppose that $\tau=\sum_{i\in K}\alpha_i>\alpha$. If $\alpha$ is a leaf of $K$, then $\tau-\alpha\in\Phi^+$. Moreover, $\tau-\alpha\notin\Phi_J^+$, since otherwise, $\alpha\in\Phi_J^+$ and thus $\tau\in\Phi_J^+$, a contradiction. Thus, we have found a $J$-star $(\alpha,\tau-\alpha)$. Assume $\alpha$ is not a leaf of $K$. Let the connected components of $K\setminus\{\alpha\}$ be $K_1,\ldots,K_m$ and let $\gamma_j=\sum_{i\in K_j}\alpha_i$. If some $\gamma_j\in\Phi_J^+$, then $J$ must contain a leaf $\alpha_j$ of $K$, meaning that $\tau-\alpha_j$ is a positive root supported on $\alpha$ so that (1) holds; and if $\gamma_j\notin\Phi_J^+$ for all $j\in[m]$, we obtain a $J$-star $(\alpha,\gamma_1,\ldots,\gamma_m)$ which sums to $\tau$ so that (2) holds.

Now suppose that some coefficient $d_i$ is at least $2$. We consider each classical type separately, but the cases are quite limited.

\

\textbf{Type $B_n$}. Say $\tau=e_a+e_b$ with $a<b$. If $e_{b}-e_{b+1}\in \Delta_J$, then (1) holds with $\alpha_j=e_b-e_{b+1}$ as $\tau$ and $\tau-\alpha_j$ have the same support. Then $e_b-e_{b+1}\notin \Delta_J$ and thus $\alpha\neq e_b-e_{b+1}$. If $\alpha=e_a-e_{a+1}$, then $(\alpha,e_{a+1}+e_b)$ is a $J$-star if $a+1<b$ (as $e_{a+1}+e_b$ is supported on $e_{b}-e_{b+1}$), and $(\alpha,2\cdot e_{a+1})$ is a $J$-star if $a+1=b$ (as $e_{a+1}$ is supported on $e_{b}-e_{b+1}$). Write $\alpha=e_i-e_{i+1}$. There are two cases: $a<i<b$ and $b<i\leq n$ (which includes the case $\alpha=e_n$). 

\textit{Case 1}: $a<i<b$. If $e_a-e_{a+1}\in \Delta_J$, then (1) holds with $\alpha_j=e_a-e_{a+1}$ and if $e_a-e_{a+1}\notin J$, then (2) holds with the $J$-star $(e_i-e_{i+1},e_a-e_i,e_{i+1}+e_b)$ when $i+1<b$, and the $J$-star $(e_i-e_{i+1},e_a-e_i,2\cdot e_b)$ when $i+1=b$.

\textit{Case 2}: $b<i$. Then (2) holds with the $J$-star $(e_i-e_{i+1},e_b-e_i,e_a+e_{i+1})$ as both $e_b-e_i$ and $e_a+e_{i+1}$ are supported on $e_b-e_{b+1}$ and thus do not belong in $\Phi_J^+$.

\

\textbf{Type $C_n$}. The analysis is largely similar so we omit some steps. Say $\tau=e_a+e_b$ with $a\leq b$. If $e_b-e_{b+1}\in \Delta_J$, then (1) holds with $\alpha_j=e_b-e_{b+1}$ as $\tau$ and $\tau-\alpha_j$ have the same support. Assume $e_b-e_{b+1}\notin \Delta_J$. If $\alpha=2e_n$, then $(2e_n,e_a-e_n,e_b-e_n)$ is a $J$-star. Say $\alpha=e_i-e_{i+1}$ with $i\geq a$ and $i\neq b$. If $i=a$, then $(e_a-e_{a+1},e_{a+1}+e_b)$ is a $J$-star; if $a<i<b$, then $(e_i-e_{i+1},e_a-e_i,e_b+e_{i+1})$ is a $J$-star when $e_a-e_{a+1}\notin \Delta_J$ and $\tau-(e_a-e_{a+1})\in\Phi^+$ is supported on $\alpha$ when $e_a-e_{a+1}\in \Delta_J$; if $i>b$, then $(e_i-e_{i+1},e_a-e_i,e_b+e_{i+1})$ is a $J$-star.

\

\textbf{Type $D_n$}. Say $\tau=e_a+e_b$ with $a< b\leq n-2$. By symmetry of the Dynkin diagram, say $\alpha=e_i-e_{i+1}$ with $i\geq a$. If $e_b-e_{b+1}\in \Delta_J$ then (1) holds with $\alpha_j=e_b-e_{b+1}$. Assume $e_b-e_{b+1}\notin \Delta_J$. If $i=a$, then $(e_i-e_{i+1},e_{i+1}-e_n,e_{i+1}+e_n)$ is a $J$-star when $b=i+1$ and $(e_i-e_{i+1},e_{i+1}+e_b)$ is a $J$-star when $b>i+1$; if $a<i<b$, then $(e_i-e_{i+1},e_a-e_i,e_{i+1}+e_b)$ is a $J$-star when $e_a-e_{a+1}\notin \Delta_J$, and (1) holds with $\alpha_j=e_a-e_{a+1}$ when $e_a-e_{a+1}\in \Delta_J$; if $i>b$, then $(e_i-e_{i+1},e_b-e_i,e_a+e_{i+1})$ is a $J$-star.

\

Type $G_2$ is trivial to check. The other exceptional types $F_4$ and $E_8$ (which covers $E_6$ and $E_7$) are done via a computer check; these checks are very quick, even for $E_8$, since we only need to iterate over the root system, not the Weyl group.
\end{proof}

\begin{proof}[Proof of \Cref{thm:intro-bp-pattern-characterization}]
First, assume that $w$ is not BP at $J$. We will then show that $w$ contains a certain $J$-star. By definition, there exists $s_\alpha\in\Supp(w^J)\cap J$ such that $s_\alpha\notin D_L(w_J)$. Since $s_\alpha\in \Supp(w^J)$, there exists $\tau\in I(w^J)$ supported on $\alpha$ (by e.g. \cite[Lemma 3.2]{gao2020boolean}). Find the minimal such $\tau$ in the root poset. Recall that $w^J$ does not have any inversions in $\Phi_J^+$. If there exists $\alpha_j\in \Delta_J$ such that $\tau-\alpha_j\in\Phi^+$ is supported on $\alpha$, then as $\alpha_j\notin I(w^J)$, we must have $\tau-\alpha_j\in I(w^J)$, contradicting the minimality of $\tau$. Thus by \Cref{lem:simple-Jstar}, there exists a $J$-star $(\alpha,c_1\gamma_1,\ldots,c_k\gamma_k)$ such that $\alpha+\sum_{i=1}^kc_i\gamma_i=\tau$. We claim that $w$ contains the $J$-star $(w_J^{-1}\alpha,c_1w_J^{-1}\gamma_1,\ldots,c_kw_J^{-1}\gamma_k)$. 

To see that this is indeed a $J$-star, note that $s_\alpha\notin D_L(w_J)$ implies that $w_J^{-1}\alpha\in\Phi^+$ and further in $\Phi_J^+$ as $w_J^{-1}\in W_J$ and $\alpha\in \Delta_J$. Each $\gamma_i$ is supported on some $\alpha_i\notin \Delta_J$ so $u\gamma_i\in\Phi^+\setminus\Phi_J^+$ for any $u\in W_J$. Moreover, $w_J^{-1}\alpha+\sum_{i\in I}c_iw_J^{-1}\gamma_i=w_J^{-1}(\alpha+\sum_{i\in I}c_i\gamma_i)\in w_J^{-1}(\Phi^+\setminus \Phi_J^+)\subset\Phi^+$. Finally, $w(w_J^{-1}\alpha)=w^J\alpha\in\Phi^+$ and $w(w_J^{-1}\alpha+\sum_{i=1}^kc_iw_J^{-1}\gamma_i)=w^J\tau\in\Phi^-$, which means that $w$ contains this $J$-star.

For the other direction, assume that $w$ contains a $J$-star at $(\beta,c_1\gamma_1,\ldots,c_k\gamma_k)$. By definition, $\beta\in\Phi_J^+$ and $w\beta\in\Phi^+$. We must have $w_J\beta\in\Phi_J^+$ since otherwise, $w^J(-w_J\beta)\in\Phi^-$, contradicting the fact that $w^J$ has no inversions in $\Phi_J^+$. As each $\gamma_i\in\Phi^+\setminus\Phi_J^+$, we have $w_J\gamma_i\in\Phi^+$ as well. Write $\tau=w_J(\beta+\sum_{i=1}^k c_i\gamma_i)\in\Phi^+\setminus\Phi_J^+$ and write $w_J\beta=\sum d_j\alpha_j$ as a linear combination of simple roots in $\Delta_J$. Since $w_J^{-1}(w_J\beta)>0$, there exists some $\alpha_j\leq\beta$ such that $w_J^{-1}\alpha_j\in\Phi^+$, i.e. $s_{\alpha_j}\notin D_L(w_J)$. Moreover, $\alpha_j\leq w_J\beta<\tau\in I(w^J)$ so $s_{\alpha_j}\in \Supp(w^J)$. This means that $w$ is not BP at $J$.
\end{proof}
\subsection{Distributivity of BP decompositions}
\label{sub:distributive}
\begin{prop}\label{prop:BP-intersection}
If $w$ is BP at $J_1$ and at $J_2$, then $w$ is BP at $J_1\cap J_2$.
\end{prop}
\begin{proof}
We use proof of contradiction. Write $J=J_1\cap J_2$. By \Cref{thm:intro-bp-pattern-characterization}, assume that $w$ contains a $J$-star at $(\beta,c_1\gamma_1,\ldots,c_k\gamma_k)$ with $k$ minimal. As a result, $\beta+\sum_{i\in I}c_i\gamma_i\notin I(w)$ for all $I\subsetneq[k]$ since otherwise, $w$ would contain a smaller $J$-star at $(\beta,c_{i_1}\gamma_{i_1},\ldots,c_{i_l}\gamma_{i_l})$ where $I=\{i_1,\ldots,i_l\}$. 

The $\gamma_i$'s are not contained in $\Delta_{J_1\cap J_2}$ and we divide them into the following three sets. Assume without loss of generality that $\gamma_1,\ldots,\gamma_m\in\Phi_{J_1}^+\setminus\Phi_{J_2}^+$, $\gamma_{m+1},\ldots,\gamma_{m+p}\in\Phi_{J_2}^+\setminus\Phi_{J_1}^+$ and $\gamma_{m+p+1},\ldots,\gamma_{k}\notin\Phi_{J_1}^+\cup\Phi_{J_2}^+$. We have $m+p+q=k\geq1$ so at least one of $m+q$ and $p+q$ is strictly positive. Assume that $p+q>0$. Then $w$ contains a $J_1$-star at $(\beta+\sum_{i=1}^mc_i\gamma_i,c_{m+1}\gamma_{m+1},\ldots,c_k\gamma_k)$, contradicting $w$ being BP at $J_1$. 
\end{proof}

\begin{defin}\label{def:non-BP-witness}
We say that $(\beta,\tau)$ is a \emph{non-BP witness} for $J$ if there exists a $J$-star $(\beta,c_1\gamma_1,\ldots,c_k\gamma_k)$ such that $\tau=\beta+\sum_{i=1}^kc_i\gamma_i$. 
\end{defin}
With \Cref{def:non-BP-witness}, \Cref{thm:intro-bp-pattern-characterization} is saying that $w$ is BP at $J$ if there does not exist a non-BP witness $(\beta,\tau)$ such that $\beta\notin I(w)$ and $\tau\in I(w)$.

\begin{lemma}\label{lem:union-helper-decomposition}
Fix $J=J_1\cup J_2$ where $J_1$ and $J_2$ are connected subgraphs of the Dynkin diagram of $\Delta$. Let $(\beta,\tau)$ be a non-BP witness for $J$ with $\beta\notin \Phi_{J_1}^+\cup\Phi_{J_2}^+$. Then we can write $\beta=\beta_1+\cdots+\beta_l$, $l\geq2$, as a sum of positive roots such that each $(\beta_i,\tau)$ is a non-BP witness for $J$ or $J_1$ or $J_2$.
\end{lemma}
\begin{proof}
We first reduce the claim to a finite check. A case analysis for classical types as in the proof of \Cref{lem:simple-Jstar} is possible. For example in type $A_{n-1}$, we can always write $\beta=\beta_1+\beta_2$ where $\beta_1\in\Phi_{J_1}^+$ and $\beta_2\in\Phi_{J_2}^+$ and it is not hard to see that $(\beta_i,\tau)$ is a non-BP witness for $J_i$, $i\in\{1,2\}$. 

For the classical types, we claim that a finite check is enough. Suppose that the lemma statement is established for type $D_r$ with $r\leq 10$. Now let $\Phi$ be a root system of type $D_n$ with $n>10$ and let $\Phi'$ be a root system of type $D_{r}$ with $r\leq 10$ to be determined later. The connected subgraph $J_1$ has one of the following forms: $\{\alpha_{i_1},\ldots,\alpha_{i_2}\}$ for some $i_1\leq i_2\leq n-2$, $\{\alpha_{i_1},\ldots,\alpha_{n-2},\alpha_{n-1}\}$ for some $i_1\leq n-2$, $\{\alpha_{i_1},\ldots,\alpha_{n-2},\alpha_n\}$ for some $i_1\leq n-2$, or $\{\alpha_{i_1},\ldots,\alpha_{n-2},\alpha_{n-1},\alpha_n\}$ for some $i_1\leq n-2$. For each case respectively, we construct $J_1'$ in the root system $\Phi'$ with the exact same form: $\{\alpha_{i_1'},\ldots,\alpha_{i_2'}\}$ for some $i_1'\leq i_2'\leq r-2$, $\{\alpha_{i_1'},\ldots,\alpha_{r-2},\alpha_{t-1}\}$ for some $i_1'\leq r-2$ and so on. Do the same thing for $J_2$ and $J_2'$. The indices $i_1'$ and $i_2'$ will be determined later. Similarly, modulo a finite number of possibilities, we construct a root $\beta'$ from $\beta$: for example, if $\beta=\alpha_{n}+\alpha_{n-1}+2\alpha_{n-2}+\cdots 2\alpha_{i_5}+\alpha_{i_5+1}+\cdots+\alpha_{i_6}$ for some $i_5\geq i_6$, then we write $\beta'=\alpha_{r}+\alpha_{r-1}+2\alpha_{r-2}+\cdots+2\alpha_{i_5'}+\alpha_{i_5'+1}+\cdots+\alpha_{i_6'}$ for some $i_5'\geq i_6'$ to be determined later. Do the same thing for $\tau$ and $\tau'$. Each $J_1,J_2,\beta,\tau$ contributes at most two indices. Choose the indices $i_1',\ldots,i_8'\in[r-2]$ so that they have the exact same order (with equalities and inequalities) as $i_1,\ldots,i_8\in[n-2]$. Take the construction $\beta'=\beta_1'+\cdots+\beta_l'$ for type $D_r$ to recover a construction $\beta=\beta_1+\cdots+\beta_l$ for type $D_n$. Note that the requirements in \Cref{def:Jstar} are all preserved in this way.

As a result, checking up to $B_9,C_9,D_{10},E_8,F_4$, and $G_2$ now suffices, which we do quickly via computer.
\end{proof}

\begin{prop}\label{prop:BP-union}
If $w$ is BP at $J_1$ and at $J_2$, then $w$ is BP at $J_1\cup J_2$.
\end{prop}
\begin{proof}
By \Cref{lem:totally-disconnected-BP}, we can assume that $J_1$ and $J_2$ are connected. We use proof by contradiction. If $w$ is not BP at $J=J_1 \cup J_2$, then there exists a non-BP witness $(\beta,\tau)$ for $J$ such that $\beta\notin I(w)$ and $\tau\in I(w)$. Pick a smallest such $\beta$ in the root poset. By \Cref{lem:union-helper-decomposition}, we have $\beta=\beta_1+\cdots+\beta_{l}$, $l\geq2$, such that each $(\beta_i,\tau)$ is a non-BP witness for $J$ or $J_1$ or $J_2$. Since $\beta\notin I(w)$, there exists $\beta_j\notin I(w)$. If $(\beta_j,\tau)$ is a non-BP witness for $J$, then the minimality of $\beta$ is contradicted; if $(\beta_j,\tau)$ is a non-BP witness for $J_i$, $i\in\{1,2\}$, then $w$ is not BP at $J_i$, a contradiction. As a result, $w$ is BP at $J$.
\end{proof}

\begin{proof}[Proof of \Cref{thm:intro-union-intersection}]
By \Cref{prop:BP-intersection,prop:BP-union} we have that $\BP(w)$ is closed under union and intersection. Thus $\BP(w)$, ordered by containment, is a sublattice of the Boolean lattice of subsets of $S$, and therefore a distributive lattice.
\end{proof}

\subsection{Properties of the BP poset}
\label{sub:bp-properties}
For a poset or preorder $P$, let $\mathcal{L}(P)$ denote the lattice of its order ideals. Consider $\BP(w)$ ordered by inclusion. It is a distributive lattice, with the meet and join operations given by intersection and union (\Cref{thm:intro-union-intersection}). By the fundamental theorem of finite distributive lattices (see for example \cite[Theorem 3.4.1]{EC1}), $\BP(w)$ is isomorphic to the lattice of order ideals of some poset.

\begin{defin}\label{def:BP-poset}
For $w\in W$, its \emph{Billey--Postnikov (BP)  preposet} $\tbp(w)$ is defined on $S$ (or, for convenience, on an indexing set for $S$) such that $\BP(w)=\mathcal{L}(\tbp(w))$. Denote its relations by $\leq_w$. The \emph{Billey--Postnikov poset} $\bp(w)$ is the quotient of $\tbp(w)$ by the equivalence relation $a\sim b$ if $a\leq_w b$ and $b\leq_w a$. The elements of $\bp(w)$ are the blocks of a set partition $\mathcal{S}_w$ of $S$.
\end{defin}
\begin{ex}
Consider $w=4231 \in \mathfrak{S}_4$ and $w'=3412$. Then $\BP(w)=\big\{\emptyset,\{1\},\{3\},\{1,3\},[3]\big\}$ and $\BP(w')=\{\emptyset,\{2\},[3]\}$; the corresponding BP posets are shown in \Cref{fig:BP-poset-example-4231-3412}. Notice that we have $1=3$ in the BP preposet $\tbp(w')$, so $\bp(w')$ has an element $\{1,3\}$ which is not a singleton.
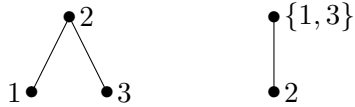
\begin{figure}[h!]
\centering
\begin{tikzpicture}[scale=1.0]
\node at (-0.5,0) {$\bullet$};
\node[left] at (-0.5,0) {$1$};
\node at (0.5,0) {$\bullet$};
\node[right] at (0.5,0) {$3$};
\node at (0,1) {$\bullet$};
\node[right] at (0,1) {$2$};
\draw(-0.5,0)--(0,1)--(0.5,0);
\end{tikzpicture}\qquad\qquad
\begin{tikzpicture}[scale=1.0]
\node at (0,0) {$\bullet$};
\node[right] at (0,0) {$2$};
\node at (0,1) {$\bullet$};
\node[right] at (0,1) {$\{1,3\}$};
\draw(0,0)--(0,1);
\end{tikzpicture}
\caption{The BP posets $\bp(w)$ and $\bp(w')$ for $w=4231$ (left) and $w'=3412$ (right).}
\label{fig:BP-poset-example-4231-3412}
\end{figure}
\end{ex}

We now use the language of \emph{closure operators} to describe the BP poset. This will later help us construct a fast algorithm to compute $\bp(w)$.
\begin{defin}\label{def:closure}
For $w\in W$ and $A\subset S$, the \emph{$w$-closure} of $A$ is \[\cl_w(A):=\bigcap\{J\in\BP(w)\:|\: J\supseteq A\}.\]
For simplicity, denote the closure $\cl_w(\{i\})$ of a singleton by $\cl_w(i)$.
\end{defin}
We now give some basic properties of $\cl_w$.
\begin{lemma}\label{lem:closure-basic}
Let $w$ be an element of the finite Weyl group $W$.
\begin{enumerate}
\item For $A\subset S$, we have $A\subset\cl_w(A)$. For $A\subset B$, we have $\cl_w(A)\subset\cl_w(B)$.
\item For $A\subset S$, $\cl_w(\cl_w(A))=\cl_w(A)$.
\item In $\tbp(w)$, $i\leq_w j$ if and only if $i\in\cl_w(j)$.
\item If $A$ is connected, $\cl_w(A)$ is also connected.
\item In $\bp(w)$, every element is covered by at most $l$ elements, where $l$ is the number of leaves in the Dynkin diagram for $W$. 
\end{enumerate}
\end{lemma}
\begin{proof}
(1) follows from the definition and (2) follows from \Cref{thm:intro-union-intersection}, which implies that $\cl_w(A) \in \BP(w)$ for any $A \subset S$. Together, (1) and (2) mean that $\cl_w:2^{S}\rightarrow 2^S$ is indeed a closure operator.

For (3), by \Cref{def:BP-poset}, $i\leq_w j$ if and only if every $J\in\BP(w)$ that contains $j$ also contains $i$, if and only if the unique minimum subset $\cl_w(j)\in\BP(w)$ that contains $j$ also contains $i$.

For (4), let $\cl_w(A)=A_1\sqcup\cdots \sqcup A_m$ be a decomposition into connected components. Since $A\subset\cl_w(A)$ and $A$ is connected, $A\subset A_k$ for some $k$. By \Cref{lem:totally-disconnected-BP}, $w$ is BP at $A_k$ so by \Cref{def:closure}, $\cl_w(A)\subset A_k$, and thus $\cl_w(A)=A_k$ which is connected.

For (5), assume to the contrary that the equivalence class of $j\in S$ is covered by the distinct equivalence classes of $i_1,\ldots,i_{l+1}$ in $\bp(w)$. Each $\cl_w(i_k)$ is connected and contains $j$, so $\cl_w(i_k)$ contains the unique path $P_k$ in the Dynkin diagram from $i_k$ to $j$. Since the Dynkin diagram is covered by the $l$ paths connecting the leaves to $j$, there must be some $P_k$ strictly contained in $P_{k'}$, so $i_k \leq_w i_{k'}$, contradicting their incomparability in $\bp(w)$.
\end{proof}

\subsection{Specialization to type $A$}
In type $A_{n-1}$, the only $J$-stars are of the form $(e_i-e_j,e_j-e_k)$, $(e_j-e_k,e_i-e_j)$ or $(e_j-e_k,e_i-e_j,e_k-e_l)$ with $i<j<k<l$. We can therefore characterize BP decompositions much more explicitly.
\begin{defin}
We say that $(w,[a,b])$, where $w\in \mathfrak{S}_n$ and $1\leq a<b\leq n$, \emph{contains the pattern} $(u,[c,d])$, where $u\in \mathfrak{S}_k$, $1\leq c<d\leq k$, if there exists indices $1\leq i_1<\cdots<i_k\leq n$ such that $w(i_x)<w(i_y)$ if and only if $u(x)<u(y)$ for all $1\leq x<y\leq k$, and $i_x\in[a,b]$ if and only if $x\in[c,d]$. We say that $(w,[a,b])$ \emph{avoids} $(u,[c,d])$ if $(w,[a,b])$ does not contain $(u,[c,d])$. We also write a pattern $(u,[c,d])$ as \[u_1\cdots u_{c-1}\underline{u_c\cdots u_d}u_{d+1}\cdots u_{k}\]
in one-line notation for simplicity.
\end{defin}

The following follows from \Cref{thm:intro-bp-pattern-characterization} after unpacking definitions.
\begin{cor}\label{cor:BP-patterns-typeA}
Let $J=\{a,\ldots,b-1\}\subset S$ be connected and $w\in \mathfrak{S}_n$. Then $w$ is BP at $J$ if and only if $(w,[a,b])$ avoids $\underline{23}1$, $3\underline{12}$, and $3\underline{14}2$. 
\end{cor}
\begin{remark}
We deduce \cite[Thm.~1.1]{alland-richmond-BP-pattern}, which states that $w$ is BP at $[n-1]\setminus\{k\}$ if and only if $w$ avoids $23|1$ and $3|12$, where the values to the left of the bar ``$|$" appear in the first $k$ indices of $w$ while the values to the right appear in the last $n-k$ indices. Based on our setup, $w$ is BP at $[n-1]\setminus\{k\}$ if and only if $w$ is BP at both $\{1,\ldots,k-1\}$ and $\{k+1,\ldots,n-1\}$. The rest follows from \Cref{cor:BP-patterns-typeA}. Note that when a connected set $J$ contains the leaf $1$ or $n-1$, we do not need to worry about the pattern $3\underline{14}2$. 
\end{remark}

\begin{lemma}\label{lem:step-by-step-closure-construction}
Let $A=\{a,\ldots,b-1\}\subset[n-1]$ be connected. If $w$ is not BP at $A$, so $(w,[a,b])$ contains one of $\underline{23}1$, $3\underline{12}$, or $3\underline{14}2$ at indices $i_1<\cdots<i_k$, then $\cl_w(A)\supseteq[i_1,i_k-1]$.
\end{lemma}
\begin{proof}
Recall that $\cl_w(A)\supseteq A$, $\cl_w(A)\in\BP(w)$, and that $\cl_w(A)=\{c,\ldots,d-1\}$ is connected, by \Cref{lem:closure-basic}. If $(w,[a,b])$ contains $\underline{23}1$ at indices $i_1<i_2<i_3$, then in order for $(w,[c,d])$ to avoid this occurrence of $\underline{23}1$, we must have $d\geq i_3$. The case of $3\underline{12}$ is the same. Now if $(w,[a,b])$ contains $3\underline{14}2$ at indices $i_1<i_2<i_3<i_4$, in order for $(w,[c,d])$ to avoid this occurrence, we need $c\leq i_1$ or $d\geq i_4$. Assume without loss of generality that $c\leq i_1$. If $d<i_4$, then $(w,[c,d])$ contains $\underline{23}1$ at indices $i_1<i_2<i_4$, contradicting $w$ being BP at $\cl_w(A)$. Therefore, $d\geq i_4$ as well.
\end{proof}

Although a direct computation of $\bp(w)$ would require checking $2^{n-1}$ parabolic decompositions, we show below that this poset can in fact be computed in polynomial time.

\begin{prop}
For any $w\in \mathfrak{S}_n$, there is a polynomial time algorithm (in $n$) to construct the BP poset $\bp(w)$ (or the BP preposet $\tbp(w)$).
\end{prop}
\begin{proof}
In light of \Cref{lem:closure-basic}(3), it suffices to compute $\cl_w(i)$ for $i=1,\ldots,n-1$. To compute $\cl_w(A)$ for connected $A=\{a,\ldots,b-1\}$, we first check if any bad pattern (cf. \Cref{cor:BP-patterns-typeA}) is contained in $(w,[a,b])$. This step takes $O(n^4)$. If no bad pattern appears, $\cl_w(A)=A$. If there is a bad pattern at indices $i_1<\cdots<i_k$, by \Cref{lem:step-by-step-closure-construction}, $\cl_w(A)\supseteq[i_1,i_k-1]\supsetneq A$ so $\cl_w(A)=\cl_w([i_1,i_k-1])$. Continue this procedure on $[i_1,i_k-1]$ which is strictly larger than $A$. At most $O(n)$ steps are needed so computing $\cl_w(i)$ needs $O(n^5)$ and the total time complexity is $O(n^6)$.
\end{proof}

\subsection{BP posets for groups of rank three}
\label{sec:rank-three}

\begin{theorem}
\label{thm:rank-three-union-intersection}
    Let $W$ be a Coxeter group of rank three and $w \in W$. If $J,K \in \BP(w)$, then $J \cap K, J \cup K \in \BP(w)$.
\end{theorem}
\begin{proof}
Let $S=\{r,s,t\}$. We always have $\emptyset, S \in \BP(w)$.

Suppose first that $J = \{r\} \in \BP(w)$. The only nontrivial case is to check that $\{r,s\} \in \BP(w)$ if $K=\{s\} \in \BP(w)$. By \Cref{prop:singleton-classification}, each of $r,s$ is either a right descent of $w$ or not in its support. If $s,r \in D_R(w)$, then $W_{\{r,s\}}$ is finite and $w_{J \cup K}=w_0(\{r,s\})$ has left descent set $\{r,s\}$, so $\{r,s\} \in \BP(w)$. If instead $r \in D_R(w)$ and $s \not \in \Supp(w)$, then $w_{J \cup K}=r$ and $\Supp(w^{J \cup K}) \cap (J \cup K) \subseteq J = D_L(w_{J \cup K})$. Reversing the roles of $r,s$ works similarly. Finally, if $r,s \not \in \Supp(w)$, then $\Supp(w^{J \cup K}) \cap (J \cup K)=\emptyset =D_L(e)$.

Now suppose that $J=\{r,s\}$. The only nontrivial case is to check that $\{s\} \in \BP(w)$ if $K=\{s,t\}$ is. Assume that $s \in \Supp(w)$, otherwise we are done. If either $r$ or $t$ is not in $\Supp(w)$ then, since $J, K \in \BP(w)$ we have $s \in D_R(w)$, so we are done by \Cref{prop:singleton-classification}. Thus we may assume that $w$ has full support, and assume for the sake of contradiction that $s \not \in D_R(w)$. In this case, by the BP condition, we have
\begin{align*}
    w_J&=\cdots rsr \\
    w_K&=\cdots tst,
\end{align*}
so $D_R(w)=\{r,t\}$ and $w_{\{r,t\}}=w_0(\{r,t\})$. Furthermore, $s$ must not commute with $r$ or $t$.

Suppose that $D_L(w_J)=\{s\}$, the case $D_L(w_J)=\{r\}$ being similar. Then $w^J=\cdots tst$. Two generators $s',s'' \in S$ with $m(s',s'')\geq 3$ are said to be \emph{interlaced} in $w$ if $s's''s'$ or $s''s's''$ appears as a (not necessarily consecutive) substring in some reduced word for $w$; as this property is invariant under braid moves, in fact it does not depend on the chosen reduced word. Inspecting $w^Jw_J$, we see that the only possible interlaced pairs are $s,r$ and $s,t$. But $w^{\{r,t\}}w_{\{r,t\}}$ would have $r,t$ interlaced if $m(r,t) \geq 3$, so $r,t$ must commute. Then we see that $D_L(w_K)=\{s\}$ (otherwise $w^K=r$ and no Coxeter moves can transform $w^Kw_K$ to $w^Jw_J$, violating Matsumoto's Theorem). Thus $w^K=\cdots rsr$ and $w_K=st \cdots st$. These reduced words for $w^Kw_K$ and $w^Jw_J$ must be connected by braid moves. No move can apply in $w_J$ or $w_K$, since $s \not \in D_R(w)$. Thus we must have $w^J \cdot s=w_0(\{s,t\})$ and $w^K \cdot s=w_0(\{r,s\})$ so that moves apply for these subwords. But this implies that $S=D_L(w)$, so $W$ is finite and $w=w_0$, contradicting the fact that $s \not \in D_R(w)$.
\end{proof}

\section{Poincar\'{e} dual Schubert classes}
\label{sec:poincare-dual}

In this section we prove \Cref{prop:intro-linear-extension}, \Cref{thm:intro-structure-constant-matrix}, and \Cref{cor:intro-bijection}.

\begin{prop}
\label{prop:singleton-classification}
Let $W$ be any Coxeter group and $w \in W$. Let $s \in S$. Then $\{s\} \in \BP(w)$ if and only if $s \in D_R(w)$ or $s \not \in \Supp(w)$.
\end{prop}
\begin{proof}
Let $J=\{s\}$. If $s \in D_R(w)$, then $w_J=s$, so $\Supp(w^J) \cap J \subseteq J \subseteq D_L(w_J)=J$. If $s \not \in \Supp(w)$, then $\Supp(w^J) \cap J \subseteq \Supp(w) \cap J = \emptyset \subseteq D_L(w_J)$. Conversely, if $s \in \Supp(w) \setminus D_R(w)$, then $w^J=w$ has $s$ in its support, but $D_L(w_J)=D_L(e)=\emptyset$. 
\end{proof}

We write $\BP_J(w_J)$ for the set of BP decompositions of $w_J$ \emph{viewed as an element of} $W_J$ (this omits the trivial BP decompositions induced by $K \subset S \setminus J$ which would be included in $\BP(w_J)$ by \Cref{prop:singleton-classification} and \Cref{thm:intro-union-intersection}). We make the analogous convention for $\bp_J(w_J)$.

\begin{prop}
\label{prop:bp-of-wJ-is-order-ideal}
Let $W$ be a finite Weyl group, $w \in W$, and $J \in \BP(w)$. Then $\bp_J(w_J)$ is the order ideal in $\bp(w)$ consisting of those blocks which are contained in $J$.
\end{prop}
\begin{proof}
We need to show that $\BP_J(w_J)=\{J \cap K \mid K \in \BP(w)\}$. Suppose $K \in \BP(w)$, then by \Cref{thm:intro-union-intersection} we have $J \cap K \in \BP(w)$. Then $J \cap K \in \BP(w_J)$ since $(w_J)_{J \cap K}=w_{J \cap K}$ and $(w_J)^{J \cap K} \leq w^{J \cap K}$. On the other hand, if $J \supset K \in \BP(w_J)$, then by \cite[Prop.~4.3]{richmond-slofstra-fiber-bundle} we have $K \in \BP(w)$, so $K$ is the intersection with $J$ of $K \in \BP(w)$.
\end{proof}

The following result was proven by Richmond and Slofstra in the case of crystallographic $W$ \cite[Thm.~3.3]{richmond-slofstra-fiber-bundle}.

\begin{prop}
\label{prop:wJ-rs}
Let $W$ be any Coxeter group and suppose $w \in W$ is rationally smooth. If $J \subset S$, then $w_J$ is rationally smooth. Furthermore, if $J \in \BP(w)$ then $w^J$ is $J$-rationally smooth.
\end{prop}
\begin{proof}
By monotonicity of Kazhdan--Lusztig polynomials \cite{elias-williamson,Slofstra-melvin}, we have $P_{w^J,w}(q) \leq P_{e,w}(q) = 1$, where $\leq$ denotes coefficientwise comparison of polynomials; thus $P_{w^J,w}(q)=1$. Thus the interval $[w^J,w]$ satisfies Deodhar's criterion for rational smoothness \cite{Dyer-nil-hecke}. Now, we have an isomorphism of posets $[e,w_J]\cong [w^J,w] $ given by $u \mapsto w^J u$, and therefore \cite{Dyer-bruhat-graph} an isomorphism of Bruhat graphs on these intervals. Thus $[e,w_J]$ also satisfies Deodhar's criterion, and we have $P_{e,w_J}(q)=1$, so $w_J$ is rationally smooth. 

Now suppose $J \in \BP(w)$. Then we have $\mathcal{P}(w)=\mathcal{P}^J(w^J)\mathcal{P}(w_J)$ by \Cref{prop:bp-iff-rgf}. We know $\mathcal{P}(w)$ is palindromic by \Cref{thm:rs-equiv-defs} and that $\mathcal{P}(w_J)$ is as well, since we have shown $w_J$ is rationally smooth. Thus $\mathcal{P}^J(w^J)$ is palindromic, so $w^J$ is $J$-rationally smooth by \Cref{thm:rs-equiv-defs}.
\end{proof}

\begin{theorem}[Richmond--Slofstra \cite{richmond-slofstra-fiber-bundle}, Thm.~3.6]
\label{thm:rs-implies-gr-bp}
Let $W$ be finite and let $w \in W$ be rationally smooth. Then $w$ has a BP decomposition induced by some maximal parabolic $J=S \setminus \{s\}$.
\end{theorem}

Combining \Cref{thm:rs-implies-gr-bp} and \Cref{prop:wJ-rs} one obtains a decomposition of a rationally smooth $X(w)$ in finite type as an iterated bundle of rationally smooth Grassmannian Schubert varieties. Using \Cref{prop:bp-of-wJ-is-order-ideal}, we see that such decompositions are indexed by linear extensions of $\bp(w)$.

\begin{proof}[Proof of \Cref{prop:intro-linear-extension}]
Let $W$ be a finite Weyl group and $w \in W$ be rationally smooth. By \Cref{thm:rs-implies-gr-bp}, we have $J=S \setminus \{s\} \in \BP(w)$ for some $s \in S$. By \Cref{thm:intro-union-intersection} and the construction of $\bp(w)$, this means that $\bp(w)$ has an order ideal the union of whose blocks is $S \setminus \{s\}$. This implies that $\bp(w)$ has $\{s\}$ as a maximal element. By \Cref{prop:wJ-rs}, $w_J$ is rationally smooth, and, by \Cref{prop:bp-of-wJ-is-order-ideal}, $\bp_J(w_J)$ equals the induced subposet of $\bp(w)$ obtained by deleting the element $\{s\}$. Repeating this, we see that all elements of $\bp(w)$ are singletons. Furthermore, linear extensions $\lambda: \bp(w) \to \{1,\ldots,|S|\}$ can be identified with sequences $\lambda^{-1}(\{1,\ldots,|S|-1\}), \ldots, \lambda^{-1}(\{1\})$ giving Grassmannian BP decompositions of the successive $w_J$'s.
\end{proof}

Although we have proven \Cref{prop:wJ-rs} for arbitrary Coxeter groups, we will see in \Cref{sec:infinite-type-counterexample} that \Cref{thm:rs-implies-gr-bp} fails in infinite type, so we obtain no such decomposition there.

\begin{ex}
Consider the smooth permutation $w=65178432$ with its BP poset shown in \Cref{fig:BP-poset-poincare-dual-example-w}. Choose a linear extension $a=(3,1,4,6,2,7,5)$, and write $J^{(i)}=J^{(i-1)}\setminus\{a_i\}$ with $J^{(0)}=S$ and $J^{(7)}=\emptyset$. We compute each $w^{(i)}=w_{J^{(i-1)}}^{J^{(i)}}$ to obtain $w^{(1)}=15623478$, $w^{(2)}=31245678$, $w^{(3)}=12374568$, $w^{(4)}=12347856$, $w^{(5)}=13245678$, $w^{(6)}=12345687$, $w^{(7)}=12346578$ so that $w=w^{(1)}w^{(2)}\cdots w^{(7)}.$ Visually, their corresponding rectangular Young's diagrams are shown in \Cref{fig:BP-poset-poincare-dual-example-w}. Here, for simplicity, at each node we draw the partition with respect to the connected component containing $a_i$ inside $J^{(i-1)}$. 
\begin{figure}[h!]
\centering
\begin{tikzpicture}[scale=0.8]
\def\r{0.3}
\def\b{0.4}
\coordinate (c3) at (0,0);
\coordinate (c4) at (4,-2);
\coordinate (c5) at (1,-3.5);
\coordinate (c6) at (4,-3.5);
\coordinate (c7) at (7,-3.5);
\coordinate (c1) at (-4,-2);
\coordinate (c2) at (-1,-2);
\draw(c3)--(c1);
\draw(c3)--(c2);
\draw(c3)--(c4)--(c5);
\draw(c6)--(c4)--(c7);
\filldraw[fill=white, draw=black] (c1) circle (\r) node {1};
\filldraw[fill=white, draw=black] (c2) circle (\r) node {2};
\filldraw[fill=white, draw=black] (c3) circle (\r) node {3};
\filldraw[fill=white, draw=black] (c4) circle (\r) node {4};
\filldraw[fill=white, draw=black] (c5) circle (\r) node {5};
\filldraw[fill=white, draw=black] (c6) circle (\r) node {6};
\filldraw[fill=white, draw=black] (c7) circle (\r) node {7};

\def\x3{1.0}
\def\y3{1.0}
\node at (\x3,\y3-3*\b) {$\cdot$};
\node at (\x3,\y3-2*\b) {$\cdot$};
\node at (\x3,\y3-1*\b) {$\cdot$};
\node at (\x3,\y3) {$\cdot$};
\node at (\x3+1*\b,\y3) {$\cdot$};
\node at (\x3+2*\b,\y3) {$\cdot$};
\node at (\x3+3*\b,\y3) {$\cdot$};
\node at (\x3+4*\b,\y3) {$\cdot$};
\node at (\x3+5*\b,\y3) {$\cdot$};
\draw(\x3,\y3-3*\b)--(\x3,\y3)--(\x3+5*\b,\y3);
\draw(\x3,\y3-2*\b)--(\x3+3*\b,\y3-2*\b)--(\x3+3*\b,\y3);
\filldraw[fill=black,opacity=0.2] (\x3,\y3-2*\b)--(\x3+3*\b,\y3-2*\b)--(\x3+3*\b,\y3)--(\x3,\y3);

\def\y1{-2.5}
\filldraw[fill=black,opacity=0.2] (-4-\b,\y1)--(-4+\b,\y1)--(-4+\b,\y1-\b)--(-4-\b,\y1-\b);
\draw (-4-\b,\y1)--(-4+\b,\y1)--(-4+\b,\y1-\b)--(-4-\b,\y1-\b)--(-4-\b,\y1);
\node at (-4-\b,\y1-\b) {$\cdot$};
\node at (-4-\b,\y1) {$\cdot$};
\node at (-4,\y1) {$\cdot$};
\node at (-4+\b,\y1) {$\cdot$};

\filldraw[fill=black,opacity=0.2] (-1-0.5*\b,\y1)--(-1+0.5*\b,\y1)--(-1+0.5*\b,\y1-\b)--(-1-0.5*\b,\y1-\b);
\draw(-1-0.5*\b,\y1)--(-1+0.5*\b,\y1)--(-1+0.5*\b,\y1-\b)--(-1-0.5*\b,\y1-\b)--(-1-0.5*\b,\y1);
\node at (-1-0.5*\b,\y1) {$\cdot$};
\node at (-1+0.5*\b,\y1) {$\cdot$};
\node at (-1-0.5*\b,\y1-\b) {$\cdot$};

\def\x4{4.5}
\def\y4{-1.5}
\filldraw[fill=black,opacity=0.2](\x4,\y4)--(\x4,\y4-\b)--(\x4+3*\b,\y4-\b)--(\x4+3*\b,\y4);
\draw(\x4+4*\b,\y4)--(\x4,\y4)--(\x4,\y4-\b)--(\x4+3*\b,\y4-\b)--(\x4+3*\b,\y4);
\node at (\x4,\y4) {$\cdot$};
\node at (\x4,\y4-\b) {$\cdot$};
\node at (\x4+\b,\y4) {$\cdot$};
\node at (\x4+2*\b,\y4) {$\cdot$};
\node at (\x4+3*\b,\y4) {$\cdot$};
\node at (\x4+4*\b,\y4) {$\cdot$};

\def\y5{-4}
\node at (1-0.5*\b,\y5) {$\cdot$};
\node at (1+0.5*\b,\y5) {$\cdot$};
\node at (1-0.5*\b,\y5-\b) {$\cdot$};
\filldraw[fill=black,opacity=0.2](1-0.5*\b,\y5)--(1+0.5*\b,\y5)--(1+0.5*\b,\y5-\b)--(1-0.5*\b,\y5-\b);
\draw(1-0.5*\b,\y5)--(1+0.5*\b,\y5)--(1+0.5*\b,\y5-\b)--(1-0.5*\b,\y5-\b)--(1-0.5*\b,\y5);

\node at (4-\b,\y5) {$\cdot$};
\node at (4,\y5) {$\cdot$};
\node at (4+\b,\y5) {$\cdot$};
\node at (4-\b,\y5-2*\b) {$\cdot$};
\node at (4-\b,\y5-\b) {$\cdot$};
\filldraw[fill=black,opacity=0.2](4-\b,\y5)--(4+\b,\y5)--(4+\b,\y5-2*\b)--(4-\b,\y5-2*\b);
\draw(4-\b,\y5)--(4+\b,\y5)--(4+\b,\y5-2*\b)--(4-\b,\y5-2*\b)--(4-\b,\y5);

\node at (7-0.5*\b,\y5) {$\cdot$};
\node at (7+0.5*\b,\y5) {$\cdot$};
\node at (7-0.5*\b,\y5-\b) {$\cdot$};
\filldraw[fill=black,opacity=0.2](7-0.5*\b,\y5)--(7+0.5*\b,\y5)--(7+0.5*\b,\y5-\b)--(7-0.5*\b,\y5-\b);
\draw(7-0.5*\b,\y5)--(7+0.5*\b,\y5)--(7+0.5*\b,\y5-\b)--(7-0.5*\b,\y5-\b)--(7-0.5*\b,\y5);
\end{tikzpicture}
\caption{The BP poset of a smooth $w=65178432$ in type $A_7$.}
\label{fig:BP-poset-poincare-dual-example-w}
\end{figure}
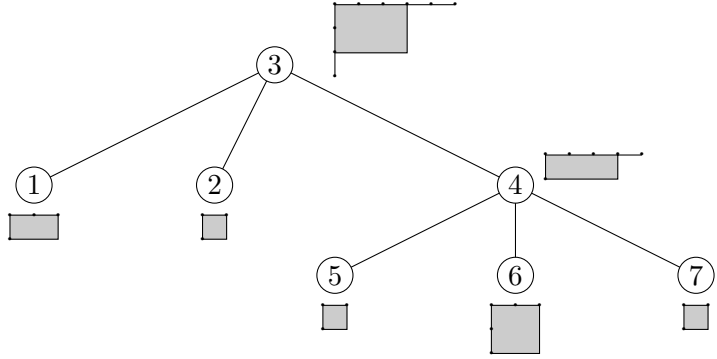

The bijection $\phi:[e,w]_k\rightarrow[e,w]_{\ell(w)-k}$ in \Cref{cor:intro-bijection} can now be constructed as follows. Given $u\in[e,w]_k$, write down $u^{(i)}=u_{J^{(i-1)}}^{J^{(i)}}$, which corresponds to a partition inside the rectangle $w^{(i)}$. Let $v^{(i)}=w_0(\Supp(w^{(i)}))\cdot u^{(i)}\cdot w_0(\Supp(w^{(i)})\setminus\{a_i\})$ be the dual shape. Then $\phi(u)=v^{(1)}v^{(2)}\cdots v^{(7)}$. An example of this duality with
\begin{align*}
u=&13624578\cdot 21345678\cdot 12364578\cdot 12345768\cdot s_2\cdot s_7\cdot e=36152784\\
v=&12534678\cdot 21345678\cdot12354678\cdot 12346857\cdot e \cdot e\cdot s_5=21548637
\end{align*}
is shown in \Cref{fig:BP-poset-poincare-dual-example-uv}.
\begin{figure}[h!]
\centering
\begin{tikzpicture}[scale=0.6]
\def\r{0.4}
\def\b{0.4}
\coordinate (c3) at (0,0);
\coordinate (c4) at (2,-2);
\coordinate (c5) at (0,-3.5);
\coordinate (c6) at (2,-3.5);
\coordinate (c7) at (4,-3.5);
\coordinate (c1) at (-3,-2);
\coordinate (c2) at (-1,-2);
\draw(c3)--(c1);
\draw(c3)--(c2);
\draw(c3)--(c4)--(c5);
\draw(c6)--(c4)--(c7);
\filldraw[fill=white, draw=black] (c1) circle (\r) node {1};
\filldraw[fill=white, draw=black] (c2) circle (\r) node {2};
\filldraw[fill=white, draw=black] (c3) circle (\r) node {3};
\filldraw[fill=white, draw=black] (c4) circle (\r) node {4};
\filldraw[fill=white, draw=black] (c5) circle (\r) node {5};
\filldraw[fill=white, draw=black] (c6) circle (\r) node {6};
\filldraw[fill=white, draw=black] (c7) circle (\r) node {7};

\def\x3{1.0}
\def\y3{1.0}
\node at (\x3,\y3-3*\b) {$\cdot$};
\node at (\x3,\y3-2*\b) {$\cdot$};
\node at (\x3,\y3-1*\b) {$\cdot$};
\node at (\x3,\y3) {$\cdot$};
\node at (\x3+1*\b,\y3) {$\cdot$};
\node at (\x3+2*\b,\y3) {$\cdot$};
\node at (\x3+3*\b,\y3) {$\cdot$};
\node at (\x3+4*\b,\y3) {$\cdot$};
\node at (\x3+5*\b,\y3) {$\cdot$};
\draw(\x3,\y3-3*\b)--(\x3,\y3)--(\x3+5*\b,\y3);
\draw(\x3,\y3-2*\b)--(\x3+\b,\y3-2*\b)--(\x3+\b,\y3-\b)--(\x3+3*\b,\y3-\b)--(\x3+3*\b,\y3);
\filldraw[fill=black,opacity=0.1] (\x3,\y3-2*\b)--(\x3+3*\b,\y3-2*\b)--(\x3+3*\b,\y3)--(\x3,\y3);
\filldraw[fill=black,opacity=0.2] (\x3,\y3-2*\b)--(\x3+\b,\y3-2*\b)--(\x3+\b,\y3-\b)--(\x3+3*\b,\y3-\b)--(\x3+3*\b,\y3)--(\x3,\y3);

\def\y1{-2.7}
\filldraw[fill=black,opacity=0.1] (-3-\b,\y1)--(-3+\b,\y1)--(-3+\b,\y1-\b)--(-3-\b,\y1-\b);
\draw(-3+\b,\y1)--(-3-\b,\y1)--(-3-\b,\y1-\b)--(-3,\y1-\b)--(-3,\y1);
\filldraw[fill=black,opacity=0.2](-3+\b,\y1)--(-3-\b,\y1)--(-3-\b,\y1-\b)--(-3,\y1-\b)--(-3,\y1);
\node at (-3-\b,\y1-\b) {$\cdot$};
\node at (-3-\b,\y1) {$\cdot$};
\node at (-3,\y1) {$\cdot$};
\node at (-3+\b,\y1) {$\cdot$};

\filldraw[fill=black,opacity=0.2] (-1-0.5*\b,\y1)--(-1+0.5*\b,\y1)--(-1+0.5*\b,\y1-\b)--(-1-0.5*\b,\y1-\b);
\draw(-1-0.5*\b,\y1)--(-1+0.5*\b,\y1)--(-1+0.5*\b,\y1-\b)--(-1-0.5*\b,\y1-\b)--(-1-0.5*\b,\y1);
\node at (-1-0.5*\b,\y1) {$\cdot$};
\node at (-1+0.5*\b,\y1) {$\cdot$};
\node at (-1-0.5*\b,\y1-\b) {$\cdot$};

\def\x4{2.8}
\def\y4{-1.7}
\filldraw[fill=black,opacity=0.1](\x4,\y4)--(\x4,\y4-\b)--(\x4+3*\b,\y4-\b)--(\x4+3*\b,\y4);
\filldraw[fill=black,opacity=0.2](\x4,\y4)--(\x4,\y4-\b)--(\x4+2*\b,\y4-\b)--(\x4+2*\b,\y4);
\draw(\x4+4*\b,\y4)--(\x4,\y4)--(\x4,\y4-\b)--(\x4+2*\b,\y4-\b)--(\x4+2*\b,\y4);
\node at (\x4,\y4) {$\cdot$};
\node at (\x4,\y4-\b) {$\cdot$};
\node at (\x4+\b,\y4) {$\cdot$};
\node at (\x4+2*\b,\y4) {$\cdot$};
\node at (\x4+3*\b,\y4) {$\cdot$};
\node at (\x4+4*\b,\y4) {$\cdot$};

\def\y5{-4.2}
\node at (0-0.5*\b,\y5) {$\cdot$};
\node at (0+0.5*\b,\y5) {$\cdot$};
\node at (0-0.5*\b,\y5-\b) {$\cdot$};
\filldraw[fill=black,opacity=0.1](0-0.5*\b,\y5)--(0+0.5*\b,\y5)--(0+0.5*\b,\y5-\b)--(0-0.5*\b,\y5-\b);
\draw(0-0.5*\b,\y5-\b)--(0-0.5*\b,\y5)--(0+0.5*\b,\y5);

\node at (2-\b,\y5) {$\cdot$};
\node at (2,\y5) {$\cdot$};
\node at (2+\b,\y5) {$\cdot$};
\node at (2-\b,\y5-2*\b) {$\cdot$};
\node at (2-\b,\y5-\b) {$\cdot$};
\filldraw[fill=black,opacity=0.1](2-\b,\y5)--(2+\b,\y5)--(2+\b,\y5-2*\b)--(2-\b,\y5-2*\b);
\draw(2+\b,\y5)--(2-\b,\y5)--(2-\b,\y5-2*\b);
\draw(2-\b,\y5-\b)--(2,\y5-\b)--(2,\y5);
\filldraw[fill=black,opacity=0.2](2-\b,\y5-\b)--(2,\y5-\b)--(2,\y5)--(2-\b,\y5);

\node at (4-0.5*\b,\y5) {$\cdot$};
\node at (4+0.5*\b,\y5) {$\cdot$};
\node at (4-0.5*\b,\y5-\b) {$\cdot$};
\filldraw[fill=black,opacity=0.2](4-0.5*\b,\y5)--(4+0.5*\b,\y5)--(4+0.5*\b,\y5-\b)--(4-0.5*\b,\y5-\b);
\draw(4-0.5*\b,\y5)--(4+0.5*\b,\y5)--(4+0.5*\b,\y5-\b)--(4-0.5*\b,\y5-\b)--(4-0.5*\b,\y5);
\end{tikzpicture}
\qquad\qquad
\begin{tikzpicture}[scale=0.6]
\def\r{0.4}
\def\b{0.4}
\coordinate (c3) at (0,0);
\coordinate (c4) at (2,-2);
\coordinate (c5) at (0,-3.5);
\coordinate (c6) at (2,-3.5);
\coordinate (c7) at (4,-3.5);
\coordinate (c1) at (-3,-2);
\coordinate (c2) at (-1,-2);
\draw(c3)--(c1);
\draw(c3)--(c2);
\draw(c3)--(c4)--(c5);
\draw(c6)--(c4)--(c7);
\filldraw[fill=white, draw=black] (c1) circle (\r) node {1};
\filldraw[fill=white, draw=black] (c2) circle (\r) node {2};
\filldraw[fill=white, draw=black] (c3) circle (\r) node {3};
\filldraw[fill=white, draw=black] (c4) circle (\r) node {4};
\filldraw[fill=white, draw=black] (c5) circle (\r) node {5};
\filldraw[fill=white, draw=black] (c6) circle (\r) node {6};
\filldraw[fill=white, draw=black] (c7) circle (\r) node {7};

\def\x3{1.0}
\def\y3{1.0}
\node at (\x3,\y3-3*\b) {$\cdot$};
\node at (\x3,\y3-2*\b) {$\cdot$};
\node at (\x3,\y3-1*\b) {$\cdot$};
\node at (\x3,\y3) {$\cdot$};
\node at (\x3+1*\b,\y3) {$\cdot$};
\node at (\x3+2*\b,\y3) {$\cdot$};
\node at (\x3+3*\b,\y3) {$\cdot$};
\node at (\x3+4*\b,\y3) {$\cdot$};
\node at (\x3+5*\b,\y3) {$\cdot$};
\draw(\x3,\y3-3*\b)--(\x3,\y3)--(\x3+5*\b,\y3);
\draw(\x3,\y3-\b)--(\x3+2*\b,\y3-\b)--(\x3+2*\b,\y3);
\filldraw[fill=black,opacity=0.1] (\x3,\y3-2*\b)--(\x3+3*\b,\y3-2*\b)--(\x3+3*\b,\y3)--(\x3,\y3);
\filldraw[fill=black,opacity=0.2] (\x3,\y3-\b)--(\x3+2*\b,\y3-\b)--(\x3+2*\b,\y3)--(\x3,\y3);

\def\y1{-2.7}
\filldraw[fill=black,opacity=0.1] (-3-\b,\y1)--(-3+\b,\y1)--(-3+\b,\y1-\b)--(-3-\b,\y1-\b);
\draw(-3+\b,\y1)--(-3-\b,\y1)--(-3-\b,\y1-\b)--(-3,\y1-\b)--(-3,\y1);
\filldraw[fill=black,opacity=0.2](-3+\b,\y1)--(-3-\b,\y1)--(-3-\b,\y1-\b)--(-3,\y1-\b)--(-3,\y1);
\node at (-3-\b,\y1-\b) {$\cdot$};
\node at (-3-\b,\y1) {$\cdot$};
\node at (-3,\y1) {$\cdot$};
\node at (-3+\b,\y1) {$\cdot$};

\filldraw[fill=black,opacity=0.1] (-1-0.5*\b,\y1)--(-1+0.5*\b,\y1)--(-1+0.5*\b,\y1-\b)--(-1-0.5*\b,\y1-\b);
\draw(-1-0.5*\b,\y1-\b)--(-1-0.5*\b,\y1)--(-1+0.5*\b,\y1);
\node at (-1-0.5*\b,\y1) {$\cdot$};
\node at (-1+0.5*\b,\y1) {$\cdot$};
\node at (-1-0.5*\b,\y1-\b) {$\cdot$};

\def\x4{2.8}
\def\y4{-1.7}
\filldraw[fill=black,opacity=0.1](\x4,\y4)--(\x4,\y4-\b)--(\x4+3*\b,\y4-\b)--(\x4+3*\b,\y4);
\filldraw[fill=black,opacity=0.2](\x4,\y4)--(\x4,\y4-\b)--(\x4+1*\b,\y4-\b)--(\x4+1*\b,\y4);
\draw(\x4+4*\b,\y4)--(\x4,\y4)--(\x4,\y4-\b)--(\x4+1*\b,\y4-\b)--(\x4+1*\b,\y4);
\node at (\x4,\y4) {$\cdot$};
\node at (\x4,\y4-\b) {$\cdot$};
\node at (\x4+\b,\y4) {$\cdot$};
\node at (\x4+2*\b,\y4) {$\cdot$};
\node at (\x4+3*\b,\y4) {$\cdot$};
\node at (\x4+4*\b,\y4) {$\cdot$};

\def\y5{-4.2}
\node at (0-0.5*\b,\y5) {$\cdot$};
\node at (0+0.5*\b,\y5) {$\cdot$};
\node at (0-0.5*\b,\y5-\b) {$\cdot$};
\filldraw[fill=black,opacity=0.2](0-0.5*\b,\y5)--(0+0.5*\b,\y5)--(0+0.5*\b,\y5-\b)--(0-0.5*\b,\y5-\b);
\draw(0-0.5*\b,\y5-\b)--(0-0.5*\b,\y5)--(0+0.5*\b,\y5)--(0+0.5*\b,\y5-\b)--(0-0.5*\b,\y5-\b);

\node at (2-\b,\y5) {$\cdot$};
\node at (2,\y5) {$\cdot$};
\node at (2+\b,\y5) {$\cdot$};
\node at (2-\b,\y5-2*\b) {$\cdot$};
\node at (2-\b,\y5-\b) {$\cdot$};
\filldraw[fill=black,opacity=0.1](2-\b,\y5)--(2+\b,\y5)--(2+\b,\y5-2*\b)--(2-\b,\y5-2*\b);
\draw(2+\b,\y5)--(2-\b,\y5)--(2-\b,\y5-2*\b);
\draw(2-\b,\y5-2*\b)--(2,\y5-2*\b)--(2,\y5-\b)--(2+\b,\y5-\b)--(2+\b,\y5);
\filldraw[fill=black,opacity=0.2](2-\b,\y5-2*\b)--(2,\y5-2*\b)--(2,\y5-\b)--(2+\b,\y5-\b)--(2+\b,\y5)--(2-\b,\y5)--(2-\b,\y5-2*\b);

\node at (4-0.5*\b,\y5) {$\cdot$};
\node at (4+0.5*\b,\y5) {$\cdot$};
\node at (4-0.5*\b,\y5-\b) {$\cdot$};
\filldraw[fill=black,opacity=0.1](4-0.5*\b,\y5)--(4+0.5*\b,\y5)--(4+0.5*\b,\y5-\b)--(4-0.5*\b,\y5-\b);
\draw(4-0.5*\b,\y5-\b)--(4-0.5*\b,\y5)--(4+0.5*\b,\y5);
\end{tikzpicture}
\caption{The duality map in $[e,w]$ for $w$ in \Cref{fig:BP-poset-poincare-dual-example-w}.}
\label{fig:BP-poset-poincare-dual-example-uv}
\end{figure}
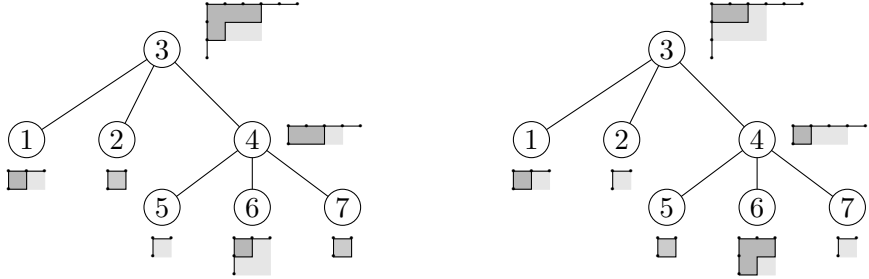
\end{ex}

\begin{remark}\label[remark]{remark:F4-bijection-not-canonical}
The idea of constructing a \emph{canonical bijection} via the matrix of Schubert structure constants does not generalize to all finite Weyl groups. Consider $w=s_2s_3s_1s_2s_1s_3s_2s_4s_3$ in type $F_4$, which indexes a smooth (not just rationally smooth) Schubert variety in the Grassmannian $G/P_J$ where $J=\{1,2,4\}$. The matrix of structure constants $c_{u,v}^w$ for $G/P_J$ between rank $4$ and rank $5$ is shown below:
\begin{center}
\begin{tabular}{c|c|c|c|}
$u\backslash v$ & $s_2s_3s_1s_2s_3$ & $s_3s_4s_1s_2s_3$ & $s_2s_3s_4s_2s_3$ \\ \hline
$s_4s_1s_2s_3$ & 1 & 1 & 1 \\ \hline
$s_3s_1s_2s_3$ & 1 & 1 & 0 \\ \hline
$s_3s_4s_2s_3$ & 0 & 1 & 1 \\ \hline
\end{tabular}.
\end{center}
It is not possible to arrange its rows and columns to make it upper triangular. In other words, there are multiple choices of $\phi:[e,w]^J_{4}\rightarrow [e,w]_5^J$ such that $c_{u,\phi(u)}^w\neq0$, and all such $\phi$'s can be extended to $G/B$, under the group element $w\cdot w_0(J)$. 
\end{remark}

\begin{proof}[Proof of \Cref{thm:intro-structure-constant-matrix}]
Let $W$ be a finite Weyl group of simply-laced type and rank $r$ and let $w \in W$ be (rationally) smooth. Choose a linear extension $\lambda: \bp(w) \to [r]$, guaranteed to exist by \Cref{prop:intro-linear-extension}, and corresponding to a chain $\emptyset = J_0 \subsetneq J_1 \subsetneq \cdots \subsetneq J_r=S$ such that for each $i$ we have that $J_{i-1} \in \BP_{J_i}(w_{J_i})$ (by \Cref{prop:wJ-rs}) and that $w_{J_{i+1}}^{J_{i}}$ and $w_{J_i}$ are $J_i$-rationally smooth and rationally smooth, respectively.

Let $J=J_{r-1}$. Assume by induction that the theorem holds for the rationally smooth element $w_{J} \in W_{J}$. Since $w^{J}$ is $J$-rationally smooth, it follows from \cite[Thm.~5.1]{richmond-slofstra-fiber-bundle} that $w^{J}=w_0(I)^{I \cap J}$ for some $I \subset S$ (this conclusion relies on the fact that $W$ is of simply-laced type, see \Cref{remark:F4-bijection-not-canonical}). This implies that $X^J(w^J)$ is itself a generalized flag variety $G_I/(G_I \cap P_{I \cap J}) \subset G/P_J$. It is well known (applying \cite{kleiman}) that generic translates of $X_I^{I \cap J}(x)$ and $X_I^{I \cap J}(y)$ do not intersect unless $ y^{\vee} \leq x$, where $y^{\vee} \coloneqq w_0(I)yw_0(I \cap J)$, and moreover that when $y^{\vee}=x$ they intersect transversely in a single point. For $u \in [e,w]_k$ and $v \in [e,w]_{\ell(w)-k}$ we therefore also have that $c_{u,v}^w=0$ unless $(u^J)^{\vee} \leq v^J$.  Thus, if we identify $[e,w]$ with $\{(x^J,x_J) \mid x^J \in W_I^{I \cap J}, x_J \in [e,w_J]\}$ and order $[e,w]_k$ (resp. $[e,w]_{\ell(w)-k}$) first with respect to a linear extension of Bruhat order on $(W_I^{I \cap J})^{\vee}$ (resp. on $W_I^{I \cap J}$) on the $x^J$ component, and then with respect to the inductively defined order on the $x_J$ component, we conclude that the matrix $(c_{u,v}^w)_{u,v}$ is upper triangular. Now, by \cite[Thm.~1.1]{richmond-multiplicative} (or as can also be seen using the degenerated Schubert product of Belkale--Kumar \cite{belkale-kumar-product}) the diagonal entries, which correspond to pairs with $(u^J)^{\vee} = v^J$, are equal to $c_{u^J, v^J}^{w^J}=1$ times a diagonal entry of $(c_{u_J,v_J}^{w_J})_{u_J,v_J}$; by induction, we conclude $(c_{u,v}^w)_{u,v}$ is upper \emph{uni}triangular. 
\end{proof}

\begin{proof}[Proof of \Cref{cor:intro-bijection}]
The corollary follows by noting that upper unitriangular matrices have a unique nonzero transversal.
\end{proof}

\section{Generalized Lehmer codes}
\label{sec:lehmer}

Recall that the classical Lehmer code $L:C_n \times C_{n-1} \times \cdots \times C_1 \to \mathfrak{S}_n$ sends $(c_1,\ldots,c_n)$ to the unique permutation $w \in \mathfrak{S}_n$ such that $c_i=|\{j>i \mid w(i)>w(j)\}|$. This map is easily checked to be order- and rank-preserving with respect to Bruhat order on $\mathfrak{S}_n$. In this section we use BP decompositions to give various generalizations (as well as obstructions to generalization) in other types and in certain parabolic quotients.
\begin{defin}
We say that a poset $P$ \emph{admits a Lehmer code} if there exists an order- and rank-preserving bijection $L:C_{a_1}\times\cdots\times C_{a_k}\rightarrow P$ for some $a_1,\ldots,a_k \in \mathbb{N}$.
\end{defin}

\begin{prop}[Lem.~22 \cite{self-dual}; Prop.~4.2 \cite{richmond-slofstra-fiber-bundle}]
\label{prop:BP-gives-subposet-product}
Let $w$ be an element of a Coxeter group $W$ and let $J \in \BP(w)$. Then the multiplication map $[e,w^J]^J \times [e,w_J] \to [e,w]$ is an order- and rank-preserving bijection.
\end{prop}

\begin{remark}
It is \emph{not} the case that the order-preserving bijection of \Cref{prop:BP-gives-subposet-product} is an order isomorphism, as one can already see by taking $w=w_0$ in $\mathfrak{S}_3$.
\end{remark}

\Cref{thm:An-Ak-lehmer} below gives a significant generalization of the classical Lehmer code to intervals below all rationally smooth elements in certain parabolic quotients of $\mathfrak{S}_n$. This result is also a key step in the construction of generalized Lehmer codes in other finite Coxeter groups.

\begin{theorem}
\label{thm:An-Ak-lehmer}
Let $W=\mathfrak{S}_n$ and let $J=\{s_j,s_{j+1},\ldots,s_{n-1}\}$ for some $1 \leq j \leq n$. Let $w \in W^J$ be $J$-rationally smooth. Then $[e,w]^J$ admits a Lehmer code.   
\end{theorem}
\begin{proof}
If $j \leq 2$ then $W^J$ is a chain under Bruhat order, so the result is clear. So suppose $j>2$. Note that by \cite[Prop.~3.4]{deodhar-parabolic} we have that $w$ is $J$-rationally smooth if and only if $u=ww_0(J)$ is rationally smooth. Let $B_1=\{i > u^{-1}(1) \mid u(i)<u(1)\}$ and $B_2=\{i<u^{-1}(1) \mid u(i)>u(1)\}$. Then at least one of $B_1$ and $B_2$ is empty, otherwise $u$ contains a $3412$ pattern $u(1)u(i_2)1u(i_1)$, where $i_1 \in B_1$ and $i_2 \in B_2$, violating the Lakshmibai--Sandhya smoothness criterion \cite{lakshmibai-sandhya}. Also note that the values of $u$ on the indices from $\{1 \leq i \leq u^{-1}(1) \mid u(i) \leq u(1)\}$ are decreasing from left to right, since otherwise $u$ would contain $4231$. 

Suppose first that $B_1 = \emptyset$. Then letting $K=\{s_2,\ldots,s_{n-1}\} \supsetneq J$ we see by \Cref{cor:BP-patterns-typeA} that $K \in \BP(u)$. Indeed, by our choice of $K$, $u$ cannot contain $\underline{23}1$ or $3\underline{14}2$ and, since $B_1 = \emptyset$, it cannot contain $3\underline{12}$. Restricting the map from \Cref{prop:BP-gives-subposet-product} then gives an order- and rank-preserving bijection $[e,u^K]^K \times [e,u_K^J]^J \to [e,u^J]^J = [e,w]^J$. Indeed, any $x \leq u$ has $x_K \leq u_K$ and hence $x_K^J \leq u_K^J$, so this restriction surjects onto $[e,w]^J$. Since $[e,u^K]^K$ is a chain and $[e,u_K^J]^J$ admits a Lehmer code by induction, the result is proved.

Now suppose that $B_2 = \emptyset$. Then by applying the above argument to $u^{-1}$, we see that $\leftindex_{K}{u} \cdot \leftindex^{K}{u}$ is a \emph{left} BP decomposition of $u$. Suppose that $j \leq u^{-1}(1)$. Then, since $J \subset D_R(u)$ we have $u^{-1}(1)=n$; since $B_2=\emptyset$ we must have that $u(1)=n$. Since $u$ avoids $4231$ we must in fact have that $u=w_0$. In this case $[e,w]^J$ has a Lehmer code obtained by restricting the classical Lehmer code to $C_n \times \cdots \times C_{n-j} \times \{0\} \times \cdots \times \{0\}$. So assume that $j > p\coloneqq u^{-1}(1)$. Then $\leftindex^{K}{u}=s_1s_2\cdots s_{p-1}$ commutes with $w_0(J)$ so 
\begin{equation}
\label{eq:commuting}
    u=\leftindex_{K}{u} \cdot \leftindex^{K}{u} = (\leftindex_{K}{u})^J w_0(J) s_1s_2\cdots s_{p-1} = (\leftindex_{K}{u})^J s_1s_2\cdots s_{p-1} w_0(J).
\end{equation}
Consider the multiplication map $[e,(\leftindex_{K}{u})^J]^J \times \leftindex^{K}{[e,\leftindex^{K}{u}]} \to [e,w]^J$. This map is length-additive, order-preserving, and injective, since $\phi: [e, \leftindex_{K}{u}] \times \leftindex^{K}{[e,\leftindex^{K}{u}]} \to [e,u]$ is. It is surjective since if $x = \phi(\leftindex_{K}{x}, \leftindex^{K}{x})  \in [e,w]^J$ then it must be that $\leftindex_{K}{x} \in W^J$; otherwise, commuting $(\leftindex_{K}{x})_J$ past $\leftindex^{K}{x}=s_1s_2 \cdots s_{q-1}$ (for $q \leq p$) as in (\ref{eq:commuting}) would show that $D_R(x) \cap J \neq \emptyset$. Now, $[e,(\leftindex_{K}{u})^J]^J$ admits a Lehmer code by induction and $\leftindex^{K}{[e,\leftindex^{K}{u}]}$ is a chain, so $[e,w]^J$ admits a Lehmer code.
\end{proof}

\begin{remark}
Note that \Cref{thm:An-Ak-lehmer} cannot be extended to arbitrary type $A$ parabolic quotients, as we see already for $(\mathfrak{S}_4)^{\{s_1,s_3\}}$, and it also cannot be extended to general types.
\end{remark}
\begin{ex}
Let $W=\mathfrak{S}_5$, $J=\{s_4\}$, and $w=52134\in W^J$. Bruhat order on $[e,w]^J$ is shown in \Cref{fig:lehmer-code-quotient}, along with the Lehmer code computed as in the proof of \Cref{thm:An-Ak-lehmer}, where we are in the case $B_2=\emptyset$, $K=\{s_2,s_3,s_4\}$, $u=52143$, $\leftindex_{K}{u}=15234$, $\leftindex^{K}{u}=s_1s_2$. In \Cref{fig:lehmer-code-quotient}, the Lehmer code for $[e,(\leftindex_{K}{u})^J]^J$ is provided in the first two digits, and the Lehmer code for $\leftindex^{K}{[e,\leftindex^{K}{u}]}$ is highlighted at the last digit. 
\begin{figure}[h!]
\centering
\begin{tikzpicture}[scale=0.6]
\def\a{1.0};
\def\b{0.7};
\def\h{3.5};
\def\r{0.2};
\newcommand\Rec[3]{
\node[align=center] at (#1,#2) {#3};
\draw(#1-\a,#2-\b+\r)--(#1-\a,#2+\b-\r);
\draw(#1-\a+\r,#2+\b)--(#1+\a-\r,#2+\b);
\draw(#1+\a,#2+\b-\r)--(#1+\a,#2-\b+\r);
\draw(#1+\a-\r,#2-\b)--(#1-\a+\r,#2-\b);
\draw (#1-\a,#2-\b+\r) arc (180:270:\r);
\draw (#1+\a-\r,#2-\b) arc (270:360:\r);
\draw (#1+\a,#2+\b-\r) arc (0:90:\r);
\draw (#1-\a+\r,#2+\b) arc (90:180:\r);
}
\Rec{0}{0}{$12345$\\$00\mathbf{0}$}
\Rec{-4*\a}{\h}{$12435$\\$10\mathbf{0}$}
\Rec{0*\a}{\h}{$13245$\\$01\mathbf{0}$}
\Rec{4*\a}{\h}{$21345$\\$00\mathbf{1}$}
\Rec{-8*\a}{2*\h}{$12534$\\$20\mathbf{0}$}
\Rec{-4*\a}{2*\h}{$14235$\\$11\mathbf{0}$}
\Rec{0*\a}{2*\h}{$21435$\\$10\mathbf{1}$}
\Rec{4*\a}{2*\h}{$23145$\\$00\mathbf{2}$}
\Rec{8*\a}{2*\h}{$31245$\\$01\mathbf{1}$}
\Rec{-8*\a}{3*\h}{$15234$\\$21\mathbf{0}$}
\Rec{-4*\a}{3*\h}{$21534$\\$20\mathbf{1}$}
\Rec{0*\a}{3*\h}{$24135$\\$10\mathbf{2}$}
\Rec{4*\a}{3*\h}{$32145$\\$01\mathbf{2}$}
\Rec{8*\a}{3*\h}{$41235$\\$11\mathbf{1}$}
\Rec{-4*\a}{4*\h}{$25134$\\$20\mathbf{2}$}
\Rec{0*\a}{4*\h}{$42135$\\$11\mathbf{2}$}
\Rec{4*\a}{4*\h}{$51234$\\$21\mathbf{1}$}
\Rec{0*\a}{5*\h}{$52134$\\$21\mathbf{2}$}

\draw(0*\a,0*\h+\b)--(-4*\a,\h-\b);
\draw(0*\a,0*\h+\b)--(0*\a,\h-\b);
\draw(0*\a,0*\h+\b)--(4*\a,\h-\b);
\draw(-4*\a,1*\h+\b)--(-8*\a,2*\h-\b);
\draw(-4*\a,1*\h+\b)--(-4*\a,2*\h-\b);
\draw(-4*\a,1*\h+\b)--(0*\a,2*\h-\b);
\draw(0*\a,1*\h+\b)--(-4*\a,2*\h-\b);
\draw(0*\a,1*\h+\b)--(4*\a,2*\h-\b);
\draw(0*\a,1*\h+\b)--(8*\a,2*\h-\b);
\draw(4*\a,1*\h+\b)--(0*\a,2*\h-\b);
\draw(4*\a,1*\h+\b)--(4*\a,2*\h-\b);
\draw(4*\a,1*\h+\b)--(8*\a,2*\h-\b);
\draw(-8*\a,2*\h+\b)--(-8*\a,3*\h-\b);
\draw(-8*\a,2*\h+\b)--(-4*\a,3*\h-\b);
\draw(-4*\a,2*\h+\b)--(-8*\a,3*\h-\b);
\draw(-4*\a,2*\h+\b)--(0*\a,3*\h-\b);
\draw(-4*\a,2*\h+\b)--(8*\a,3*\h-\b);
\draw(0*\a,2*\h+\b)--(-4*\a,3*\h-\b);
\draw(0*\a,2*\h+\b)--(0*\a,3*\h-\b);
\draw(0*\a,2*\h+\b)--(8*\a,3*\h-\b);
\draw(4*\a,2*\h+\b)--(0*\a,3*\h-\b);
\draw(4*\a,2*\h+\b)--(4*\a,3*\h-\b);
\draw(8*\a,2*\h+\b)--(4*\a,3*\h-\b);
\draw(8*\a,2*\h+\b)--(8*\a,3*\h-\b);
\draw(-8*\a,3*\h+\b)--(-4*\a,4*\h-\b);
\draw(-8*\a,3*\h+\b)--(4*\a,4*\h-\b);
\draw(-4*\a,3*\h+\b)--(-4*\a,4*\h-\b);
\draw(-4*\a,3*\h+\b)--(4*\a,4*\h-\b);
\draw(0*\a,3*\h+\b)--(-4*\a,4*\h-\b);
\draw(0*\a,3*\h+\b)--(0*\a,4*\h-\b);
\draw(4*\a,3*\h+\b)--(0*\a,4*\h-\b);
\draw(8*\a,3*\h+\b)--(0*\a,4*\h-\b);
\draw(8*\a,3*\h+\b)--(4*\a,4*\h-\b);
\draw(-4*\a,4*\h+\b)--(0*\a,5*\h-\b);
\draw(0*\a,4*\h+\b)--(0*\a,5*\h-\b);
\draw(4*\a,4*\h+\b)--(0*\a,5*\h-\b);
\end{tikzpicture}
\caption{The Lehmer code for $[e,52134]^{\{s_4\}}$ as constructed in \Cref{thm:An-Ak-lehmer}.}
\label{fig:lehmer-code-quotient}
\end{figure}
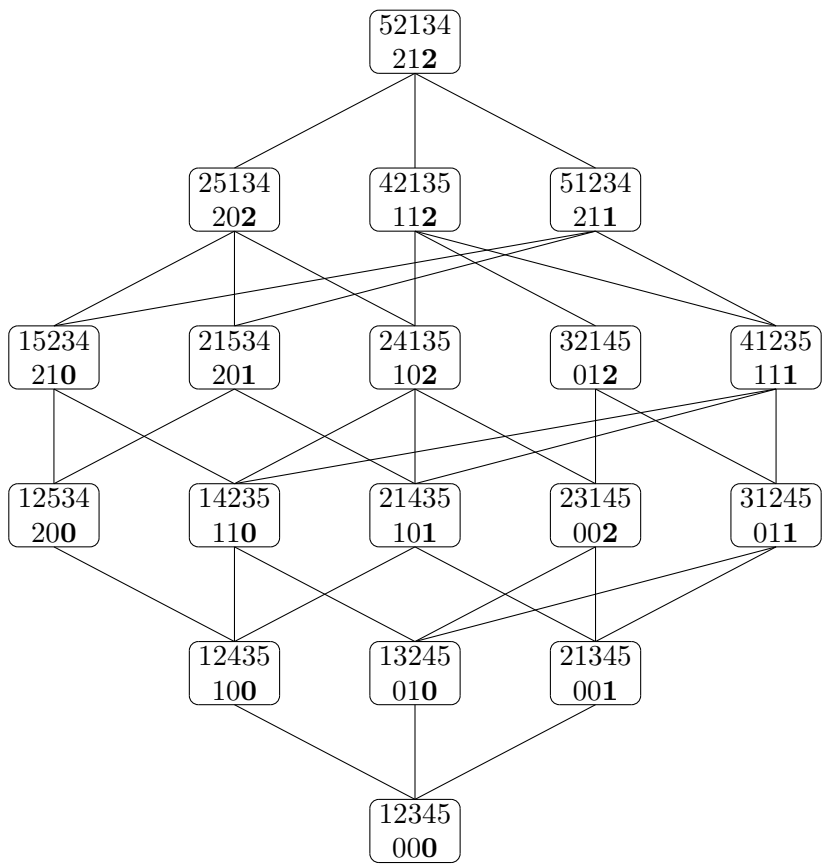
\end{ex}

The following (\emph{almost} type-uniform) result is key to our proof of \Cref{thm:intro-lehmer-code}; it reduces the Lehmer code question for rationally smooth $w$, except in a few edge cases, to the question for $w=w_0$.

\begin{prop}
\label{prop:min-is-w0}
Let $W$ be a finite Coxeter group not having an irreducible factor of type $E$. Suppose that $w \in W$ is minimal, under Bruhat order, among all rationally smooth elements of $W$ which do not admit a Lehmer code. Then $w=w_0(\Supp(w))$.
\end{prop}
\begin{proof}
Let $K=\Supp(w)$. First observe that $K$ must be connected, since otherwise $[e,w]$ would be a product of rationally smooth intervals $[e,w_{K'}]$ for $K'$ a connected component of $K$. But each of these has a Lehmer code by hypothesis, yielding a Lehmer code for $[e,w]$.

Now let $J = K \setminus \{s\} \in \BP(w)$ be a Grassmannian BP decomposition, with $s$ a leaf of $K$, which without loss of generality we can assume exists by \cite[Thm.~5.1]{richmond-slofstra-fiber-bundle}; if this did not exist for $w$, then we could apply the inversion automorphism of Bruhat order on $W_K$, which preserves rational smoothness. Then, since $w_J < w$, and since $w_J$ is rationally smooth by \Cref{prop:wJ-rs}, we have by our assumption of minimality that $[e,w_J]$ admits a Lehmer code $L_1: C_{a_1} \times \cdots \times C_{a_k} \to [e,w_J]$. Clearly, then, if $[e,w^J]^J$ admits a Lehmer code $L_2: C_{b_1} \times \cdots \times C_{b_l} \to [e,w^J]^J$ then the composition
\[
m \circ (L_2 \times L_1): (C_{b_1} \times \cdots \times C_{b_l}) \times (C_{a_1} \times \cdots \times C_{a_k}) \to [e,w^J]^J \times [e,w_J] \to [e,w]
\]
yields a Lehmer code of $[e,w]$, where $m$ is the multiplication map of \Cref{prop:BP-gives-subposet-product}. So, by our assumption on $w$, $[e,w^J]^J$ must not admit a Lehmer code. 

We now aim to argue that $w^J=w_0(I)^{I \cap J}$ for some $I \subset K$ with $s \in I$. By \Cref{prop:wJ-rs}, we know that $w^J$ is $J$-rationally smooth in $W_K$. So, by \cite[Thm.~5.1]{richmond-slofstra-fiber-bundle},  the alternative is that $w^J$ falls into one of the following exceptional cases (numbering the simple reflections as in \Cref{fig:cox-diagrams}):
\begin{itemize}
    \item[(1)] $W_{\Supp(w^J)}$ is of type $B_n$ with either
        \begin{itemize}
            \item[$\bullet$] $J=\Supp(w^J) \setminus \{s_1\}$ and $w^J=s_js_{j+1}\cdots s_ns_{n-1} \cdots s_1$ for some $1 < j \leq n$, or
            \item[$\bullet$] $J=\Supp(w^J) \setminus \{s_n\}$ with $n \geq 2$ and $w^J=s_1\cdots s_n$.
        \end{itemize}
    \item[(2)] $W_{\Supp(w^J)}$ is of type $F_4$ with either
        \begin{itemize}
            \item[$\bullet$] $J=\Supp(w^J) \setminus \{s_1\}$ and $w^J=s_4s_3s_2s_1$, or
            \item[$\bullet$] $J=\Supp(w^J) \setminus \{s_4\}$ and $w^J=s_1s_2s_3s_4$.
        \end{itemize}
    \item[(3)] $W_{\Supp(w^J)}$ is of type $I_2(m)$ and $2 \leq \ell(w^J) \leq m-2$.
\end{itemize}
However, for each of these exceptional cases, it is not hard to check that $[e,w^J]^J$ is a chain, so in particular admits a Lehmer code. Thus it must be that $w^J=w_0(I)^{I \cap J}$ for some $I \subset K$. Furthermore, we must have $s \in I$, otherwise $w^J=e$, so $[e,w^J]^J$ would admit a (trivial) Lehmer code.

We know now that our BP decomposition is of the form $w=w_0(I)^{I \cap J} \cdot w_J$. If $I'$ is a connected component of $I$ not containing $s$, then $I' \subset J$, and replacing $I$ with $I \setminus I'$ does not change the value of $w_0(I)^{I \cap J}$, so we may assume that $I$ is connected. Given any $s' \in I$, there is a non-backtracking path $s'=s_{i_1}, s_{i_2}, \ldots, s_{i_k}=s$ from $s'$ to $s$ in the Coxeter diagram of $I$. This yields an element $x=s_{i_1}\cdots s_{i_k} \in W_I^{I \cap J}$ with $s'$ in its support. Since $x \leq w_0(I)^{I \cap J}$, we conclude that $\Supp(w_0(I)^{I \cap J})=I$. Therefore, by \Cref{def:BP}, the fact that $w_0(I)^{I \cap J} \cdot w_J$ is a BP decomposition yields:
\[
\Supp(w_0(I)^{I \cap J}) \cap J = I \cap J \subset D_L(w_J).
\]
This implies that $w_J=w_0(I \cap J)\cdot u$, where $u \in \leftindex^{I\cap J}{W}_J$ and that $w=w_0(I)\cdot u$. This is a \emph{left} BP decomposition of $w$ with respect to $I$, since $D_R(w_0(I)) = I \supset I \cap \Supp(u)$ and $u^{-1}$ is $I$-rationally smooth. If $I=K$, then $u \in \leftindex^{J}{W}_J = \{e\}$ so $w=w_0(K)$ and we are done. Otherwise we have $I \subsetneq K$; thus $w_0(I) < w$ admits a Lehmer code by assumption (the longest element of a parabolic subgroup is always rationally smooth). 

We have reduced to the case where $I, J \subsetneq K$ are connected, $K$ is not of type $E$, $J=K \setminus \{s\}$ for a leaf $s$ of $K$, and $s \in I$. By inspection of the possible Coxeter diagrams (see \Cref{fig:cox-diagrams}), we see that at least one of the following holds:
\begin{itemize}
    \item $J$ is of type $A$ with $I \cap J$ containing a leaf of $J$, so $[e,w]$ has a Lehmer code induced by the BP decomposition $w=w_0(I)u$ using the Lehmer code for $\leftindex^{I\cap J}[e,u]$ from \Cref{thm:An-Ak-lehmer}; or,
    \item $I$ is of type $A$ with $I \cap J$ containing a leaf of $I$, so $[e,w]$ has a Lehmer code induced by the BP decomposition $w=w_0(I)^{I \cap J}w_J$ using the Lehmer code for $[e,w_0(I)^{I \cap J}]^{I \cap J}$ from \Cref{thm:An-Ak-lehmer}; or,
    \item $(I \cap J, J)$ are of type $(B_{n-1},B_n)$, yielding a Lehmer code since the parabolic quotient $W_J^{I \cap J}$ is a chain (in fact this case is only needed for $K$ of type $F_4$).
\end{itemize}
\end{proof}

\begin{remark}
\label[remark]{rmk:type-E}
    Our proof of \Cref{prop:min-is-w0} goes through in type $E$ except for very few cases, such as $(I \cap J, I, J, K)$ having types $(D_4,D_5,D_5,E_6)$. 
\end{remark}

\begin{figure}
\centering
\begin{tikzpicture}[scale=0.7]
\draw(3,0)--(0,0);
\draw (8,0)--(5,0);
\node[draw,shape=circle,fill=black,scale=0.5][label=below: {$1$}] at (0,0) {};
\node[draw,shape=circle,fill=black,scale=0.5][label=below: {$2$}] at (1,0) {};
\node[draw,shape=circle,fill=black,scale=0.5][label=below: {$3$}] at (2,0) {};
\node[draw,shape=circle,fill=black,scale=0.5][label=below: {}] at (3,0) {};
\node[draw,shape=circle,fill=black,scale=0.5][label=below: {}] at (5,0) {};
\node[draw,shape=circle,fill=black,scale=0.5][label=below: {}] at (6,0) {};
\node[draw,shape=circle,fill=black,scale=0.5][label=below: {$n{-}1$}] at (7,0) {};
\node[draw,shape=circle,fill=black,scale=0.5][label=below: {$n$}] at (8,0) {};
\node at (4,0) {$\cdots$};
\node at (-1,0) {$A_n$};

\draw(0,-2)--(3,-2);
\draw(5,-2)--(7,-2);
\draw(7,-2)--(8,-2);
\node[draw,shape=circle,fill=black,scale=0.5][label=below: {$1$}] at (0,-2) {};
\node[draw,shape=circle,fill=black,scale=0.5][label=below: {$2$}] at (1,-2) {};
\node[draw,shape=circle,fill=black,scale=0.5][label=below: {$3$}] at (2,-2) {};
\node[draw,shape=circle,fill=black,scale=0.5][label=below: {}] at (3,-2) {};
\node[draw,shape=circle,fill=black,scale=0.5][label=below: {}] at (5,-2) {};
\node[draw,shape=circle,fill=black,scale=0.5][label=below: {}] at (6,-2) {};
\node[draw,shape=circle,fill=black,scale=0.5][label=below: {$n{-}1$}] at (7,-2) {};
\node[draw,shape=circle,fill=black,scale=0.5][label=below: {$n$}] at (8,-2) {};
\node at (4,-2) {$\cdots$};
\node at (7.5,-1.5) {$4$};
\node at (-1,-2) {$B_n$};

\draw(0,-4)--(3,-4);
\draw(5,-4)--(8,-4);
\draw(7,-3.2)--(7,-4);
\node at (4,-4) {$\cdots$};
\node[draw,shape=circle,fill=black,scale=0.5][label=below: {$1$}] at (0,-4) {};
\node[draw,shape=circle,fill=black,scale=0.5][label=below: {$2$}] at (1,-4) {};
\node[draw,shape=circle,fill=black,scale=0.5][label=below: {$3$}] at (2,-4) {};
\node[draw,shape=circle,fill=black,scale=0.5][label=below: {}] at (3,-4) {};
\node[draw,shape=circle,fill=black,scale=0.5][label=below: {}] at (5,-4) {};
\node[draw,shape=circle,fill=black,scale=0.5][label=below: {}] at (6,-4) {};
\node[draw,shape=circle,fill=black,scale=0.5][label=below: {$n{-}2$}] at (7,-4) {};
\node[draw,shape=circle,fill=black,scale=0.5][label=above: {$n{-}1$}] at (8,-4) {};
\node[draw,shape=circle,fill=black,scale=0.5][label=left: {$n$}] at (7,-3.2) {};
\node at (-1,-4) {$D_n$};

\draw(0,-6)--(1,-6);
\draw(1,-6)--(2,-6);
\node[draw,shape=circle,fill=black,scale=0.5][label=below: {$1$}] at (0,-6) {};
\node[draw,shape=circle,fill=black,scale=0.5][label=below: {$2$}] at (1,-6) {};
\node[draw,shape=circle,fill=black,scale=0.5][label=below: {$3$}] at (2,-6) {};
\node at (1.5,-5.5) {$5$};
\node at (-1,-6) {$H_3$};

\draw(0,-8)--(1,-8);
\draw(2,-8)--(3,-8);
\draw(1,-8)--(2,-8);
\node[draw,shape=circle,fill=black,scale=0.5][label=below: {$1$}] at (0,-8) {};
\node[draw,shape=circle,fill=black,scale=0.5][label=below: {$2$}] at (1,-8) {};
\node[draw,shape=circle,fill=black,scale=0.5][label=below: {$3$}] at (2,-8) {};
\node[draw,shape=circle,fill=black,scale=0.5][label=below: {$4$}] at (3,-8) {};
\node at (2.5,-7.5) {$5$};
\node at (-1,-8) {$H_4$};

\draw(12,0)--(16,0);
\draw(14,0)--(14,0.8);
\node[draw,shape=circle,fill=black,scale=0.5][label=below: {$1$}] at (12,0) {};
\node[draw,shape=circle,fill=black,scale=0.5][label=below: {$2$}] at (13,0) {};
\node[draw,shape=circle,fill=black,scale=0.5][label=below: {$3$}] at (14,0) {};
\node[draw,shape=circle,fill=black,scale=0.5][label=below: {$4$}] at (15,0) {};
\node[draw,shape=circle,fill=black,scale=0.5][label=below: {$5$}] at (16,0) {};
\node[draw,shape=circle,fill=black,scale=0.5][label=right: {$6$}] at (14,0.8) {};
\node at (11,0) {$E_6$};

\draw(12,-2)--(17,-2);
\draw(14,-2)--(14,-1.2);
\node[draw,shape=circle,fill=black,scale=0.5][label=below: {$1$}] at (12,-2) {};
\node[draw,shape=circle,fill=black,scale=0.5][label=below: {$2$}] at (13,-2) {};
\node[draw,shape=circle,fill=black,scale=0.5][label=below: {$3$}] at (14,-2) {};
\node[draw,shape=circle,fill=black,scale=0.5][label=below: {$4$}] at (15,-2) {};
\node[draw,shape=circle,fill=black,scale=0.5][label=below: {$5$}] at (16,-2) {};
\node[draw,shape=circle,fill=black,scale=0.5][label=below: {$6$}] at (17,-2) {};
\node[draw,shape=circle,fill=black,scale=0.5][label=right: {$7$}] at (14,-1.2) {};
\node at (11,-2) {$E_7$};

\draw(12,-4)--(18,-4);
\draw(16,-4)--(16,-3.2);
\node[draw,shape=circle,fill=black,scale=0.5][label=below: {$1$}] at (12,-4) {};
\node[draw,shape=circle,fill=black,scale=0.5][label=below: {$2$}] at (13,-4) {};
\node[draw,shape=circle,fill=black,scale=0.5][label=below: {$3$}] at (14,-4) {};
\node[draw,shape=circle,fill=black,scale=0.5][label=below: {$4$}] at (15,-4) {};
\node[draw,shape=circle,fill=black,scale=0.5][label=below: {$5$}] at (16,-4) {};
\node[draw,shape=circle,fill=black,scale=0.5][label=below: {$6$}] at (17,-4) {};
\node[draw,shape=circle,fill=black,scale=0.5][label=below: {$7$}] at (18,-4) {};
\node[draw,shape=circle,fill=black,scale=0.5][label=right: {$8$}] at (16,-3.2) {};
\node at (11,-4) {$E_8$};

\draw(12,-6)--(13,-6);
\draw(14,-6)--(15,-6);
\draw(13,-6)--(14,-6);
\node[draw,shape=circle,fill=black,scale=0.5][label=below: {$1$}] at (12,-6) {};
\node[draw,shape=circle,fill=black,scale=0.5][label=below: {$2$}] at (13,-6) {};
\node[draw,shape=circle,fill=black,scale=0.5][label=below: {$3$}] at (14,-6) {};
\node[draw,shape=circle,fill=black,scale=0.5][label=below: {$4$}] at (15,-6) {};
\node at (13.5,-5.5) {$4$};
\node at (11,-6) {$F_4$};

\draw(12,-8) -- (13,-8);
\node at (12.5,-7.5) {$m$};
\node[draw,shape=circle,fill=black,scale=0.5][label=below: {$1$}] at (12,-8) {};
\node[draw,shape=circle,fill=black,scale=0.5][label=below: {$2$}] at (13,-8) {};
\node at (11,-8) {$I_2(m)$};
\end{tikzpicture}
\caption{The Coxeter diagrams for the finite irreducible Coxeter groups.}
\label{fig:cox-diagrams}
\end{figure}
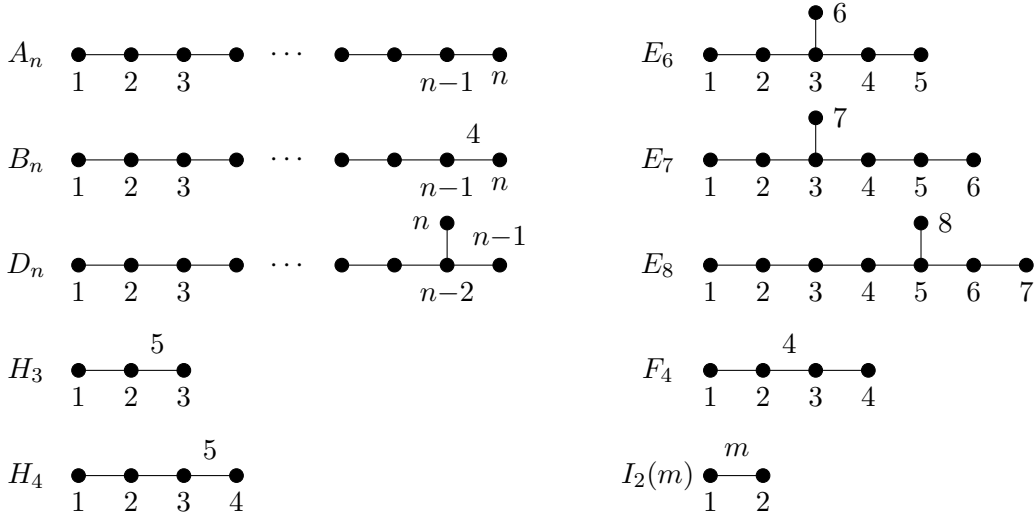

We are now ready to prove \Cref{thm:intro-lehmer-code}.

\begin{proof}[Proof of \Cref{thm:intro-lehmer-code}]
\Cref{prop:min-is-w0} allows us to reduce almost entirely to the case $w$ is the longest element of a parabolic subgroup. If $W$ has type $A_n$ or $B_n$ then $[e,w_0]$ has a Lehmer code since the parabolic quotients of types $A_n/A_{n-1}$ and $B_n/B_{n-1}$ are chains. In type $D_n$ a Lehmer code for $[e,w_0]$ was constructed by Bolognini and Sentinelli \cite{bolognini-sentinelli}. In type $H_3$ the authors of \cite{cubulation} computed a Lehmer code for $[e,w_0]$. Since all parabolic subgroups of these $W$ are also of types known to have Lehmer codes, we have by \Cref{prop:min-is-w0} that in fact all rationally smooth intervals in these types admit Lehmer codes.

On the other hand, in types $F_4$ and $H_4$ it is known \cite{cubulation,sentinelli-zatti} that $[e,w_0]$ does \emph{not} admit a Lehmer code. However, since all proper parabolic subgroups of these groups are known to admit Lehmer codes, \Cref{prop:min-is-w0} implies that all proper rationally smooth intervals $[e,w]$ do admit Lehmer codes. 
\end{proof}

\section{Rationally smooth elements in infinite type}
\label{sec:infinite-type-counterexample}

In light of the utility of \Cref{thm:rs-implies-gr-bp} in the understanding of rationally smooth Schubert varieties in finite type, it is natural to hope that it could be extended to infinite type. Whether this is possible was posed as a question by Oh and Richmond \cite[Q.~8.1]{oh-richmond-chapter} and an affirmative answer was conjectured by Richmond and Slofstra \cite[Conj.~6.5]{richmond-slofstra-fiber-bundle}. We recall their \Cref{conj:infinite-type-grassmannian-bp} here:

\begin{conj}[Richmond--Slofstra \cite{richmond-slofstra-fiber-bundle}; Oh--Richmond \cite{oh-richmond-chapter}]
Let $W$ be any Coxeter group and suppose that $w \in W$ is rationally smooth. Then $w$ has a Grassmannian BP decomposition.
\end{conj}

\Cref{conj:infinite-type-grassmannian-bp} is known to hold when $W$ is finite \cite{richmond-slofstra-fiber-bundle, ryan, wolper}, of affine type $\widetilde{A}_n$ (see \cite{billey-crites} and the discussion in \cite[\S~6.1]{richmond-slofstra-fiber-bundle}), or in a certain large family of groups of indefinite type \cite{richmond-slofstra-triangle}. In \Cref{thm:counterexample} we give the first known counterexample. The conjecture clearly holds for any element $w$ in any Coxeter group such that $|\Supp(w)|=2$, since $\{s\} \in \BP(w)$ for any $s \in D_R(w)$. We should thus consider groups $W$ of rank three. Many of these are either finite, or of type $\widetilde{A}_2$, or among the groups from \cite{richmond-slofstra-triangle}. One of the ``smallest" remaining possible sources of a counterexample is the affine Weyl group of type $\widetilde{C}_2$.

\begin{theorem}
\label{thm:counterexample}
\Cref{conj:infinite-type-grassmannian-bp} fails for the affine Weyl group $W$ of type $\widetilde{C}_2$.
\end{theorem}
\begin{proof}
We sketch how one can search for a counterexample and then produce one. The group $W$ has simple generating set $S=\{r,s,t\}$ with $(rs)^4=(st)^4=(rt)^2=e$. We want to construct a counterexample $w \in W$. By \Cref{thm:rank-three-union-intersection}, we must have $|D_R(w)|=1$, and by the observation above, we must have $\Supp(w)=S$. Suppose $D_R(w)=\{r\}$. 

As $r$ and $t$ commute, we must have $t \not \in D_R(wr)$, hence $s \in D_R(wr)$. Suppose that $r \not \in D_R(wrs)$, in which case we must have $t \in D_R(wrs)$. Constructing a reduced word for $w$ from right to left, we see that our descent conditions imply that $w \leq_L (stsr)^k$ for some $k$. In order for such $w$ not to have a Grassmannian BP decomposition, we need $\ell(w) \geq 5$, but any element of this form has at least four lower Bruhat covers, so is not rationally smooth.

Thus our counterexample $w$ must have $r \in D_R(wrs)$. The condition $D_R(w)=\{r\}$ then implies that $w$ has a length-additive expression of the form $u strsr$ for some $u \in W$. Taking $u\in \{e,r,t\}$ does not give the desired counterexample since the resulting elements $w$ are not rationally smooth, having more than $|S|=3$ lower Bruhat covers. 

We are thus led to consider $u=sr$ and $w=srstrsr$; this is our counterexample. Indeed, we have Grassmannian parabolic decompositions:
\begin{align*}
    w&=srst \cdot rsr & \text{ for $J=\{r,s\}$,} \\
    w&=srstrs \cdot r & \text{ for $J=\{r,t\}$,} \\
    w&=w \cdot e & \text{ for $J=\{s,t\}$.}
\end{align*}
In each case we see that $J \not \in \BP(w)$. A direct enumeration shows that $\mathcal{P}^{\{r\}}(w^{\{r\}})=1+2q+3q^2+4q^3+3q^4+2q^5+q^6$, so $w$ is rationally smooth by \Cref{prop:bp-iff-rgf} and \Cref{thm:rs-equiv-defs}.
\end{proof}

\begin{remark}
Note that affine Weyl groups are crystallographic, so \Cref{conj:infinite-type-grassmannian-bp} fails even in this most geometrically natural setting. 
\end{remark}

\section*{Acknowledgements}
We are grateful to Allen Knutson for helpful comments. C.G. was supported by the National Science Foundation under award no. DMS-2452032 and by a travel grant from the Simons Foundation. Y.G. was partially supported by NSFC Grant no. 12471309. We are also thankful for the excellent collaborative environment provided by the Sydney Mathematical Research Institute, which we visited in May 2025 for the program ``Modern Perspectives in Representation Theory".

We are grateful to Jan Petr, Anne Thomas, and Pavel Turek for pointing out an error in the definition of ``Lehmer code" in an earlier version.

\bibliographystyle{plain}
\bibliography{main}
\end{document}